\documentclass[oneside,english]{amsart}
\usepackage[T1]{fontenc}
\usepackage[latin9]{inputenc}
\usepackage{mathrsfs}
\usepackage{amsthm}
\usepackage{amstext}
\usepackage{esint}

\def\hpq0{h^{p,q}_{\leq 0}}
\def\Hpq0{\H_{\leq 0}^{p,q}}

\def\dbar{\bar\partial}
\def\ddbar{\partial\dbar}

\def\R{{\mathbb R}}

\def\C{{\mathbb C}}

\def\M{{\mathcal M}}

\def\F{{\mathcal F}}

\def\D{{\mathcal D}}

\def\h{{\mathcal{H}}}
\def\e{{\mathcal E}}

\def\Re{{\rm Re\,  }}
\def\Im{{\rm Im\,  }}

\makeatletter
\numberwithin{equation}{section}
\numberwithin{figure}{section}
\theoremstyle{plain}
\newtheorem{thm}{\protect\theoremname}[section]
  \theoremstyle{plain}
  \newtheorem{cor}[thm]{\protect\corollaryname}
  \theoremstyle{plain}
  \newtheorem{prop}[thm]{\protect\propositionname}
  \theoremstyle{plain}
  \newtheorem{lem}[thm]{\protect\lemmaname}
  \theoremstyle{remark}

\def\makebbb#1{
    \expandafter\gdef\csname#1\endcsname{
        \ensuremath{\Bbb{#1}}}
}\makebbb{R}\makebbb{N}\makebbb{Z}\makebbb{C}\makebbb{H}\makebbb{E}\makebbb{H}\makebbb{P}\makebbb{B}\makebbb{Q}\makebbb{E}

\usepackage{babel}

\usepackage{babel}

\makeatother

\usepackage{babel}

\makeatother

\usepackage{babel}
  \providecommand{\corollaryname}{Corollary}
  \providecommand{\lemmaname}{Lemma}
  \providecommand{\propositionname}{Proposition}
  \providecommand{\remarkname}{Remark}
\providecommand{\theoremname}{Theorem}

\begin{document}

\title{Convexity of the K-energy on the space of K\"ahler metrics and uniqueness of extremal metrics}

\author{Robert J. Berman, Bo Berndtsson}

\curraddr{Department of Mathemical Sciences, Chalmers University of Technology
and University of Gothenburg. SE-412 96 Gothenburg, Sweden}

\email{robertb@chalmers.se, bob@chalmers.se}
\begin{abstract}
We establish the convexity of Mabuchi's K-energy functional along
weak geodesics in the space of K\"ahler potentials on a compact K\"ahler
manifold, thus confirming a conjecture of Chen and give some applications
in K\"ahler geometry, including a proof of the uniqueness of constant
scalar curvature metrics (or more generally extremal metrics) modulo
automorphisms. The key ingredient is a new local positivity property
of weak solutions to the homogenuous Monge-Amp\`ere equation on a product
domain, whose proof uses plurisubharmonic variation of  Bergman kernels. 

\tableofcontents{}
\end{abstract}
\maketitle

\section{\label{sec:Introduction}Introduction}

Let $X$ be an $n-$dimensional compact complex manifold equipped
with a K\"ahler form $\omega_{0}.$ In the seminal work of Calabi \cite{ca,ca-2}
the problem of finding a canonical K\"ahler metric in the corresponding
cohomology class $[\omega_{0}]\in H^{2}(X,\R)$ was proposed; in particular
a metric with constant scalar curvature. As later shown by Mabuchi
\cite{mab0} such metrics are the critical points of a certain functional
on the space of K\"ahler metrics in $[\omega_{0}]$ called the \emph{K-energy}
or the \emph{Mabuchi functional, }which we will denote by $\mathcal{M},$
defined as follows. First recall that the space of all K\"ahler metrics
in $[\omega]$ may be identified with the space $\mathcal{H}(X,\omega)$
of all K\"ahler potentials, modulo constants, i.e. the space of all
functions $u$ on $X$ such that 
\[
\omega_{u}:=\omega+dd^{c}u,\,\,\,\,\,\,(dd^{c}:=\frac{i}{2\pi}\partial\bar{\partial})
\]
 is positive, i.e. defines a K\"ahler form on $X.$ The space $\mathcal{H}(X,\omega)$
admits a natural Riemannian metric $g$ (of non-positive sectional
curvature) that we will refer to as the \emph{Mabuchi metric \cite{mab-1},}
where the squared norm of a tangent vector $v\in C^{\infty}(X)$ at
$u$ is defined by 
\begin{equation}
g_{|u}(v,v):=\int_{X}v^{2}\omega_{u}^{n}\label{eq:mab metric intro}
\end{equation}
Now the Mabuchi functional $\mathcal{M}$ on the infinite dimensional
Riemannian manifold $\mathcal{H}(X,\omega)$ is uniquely defined,
modulo an additive constant, by the property that is gradient is the
normalized scalar curvature of the corresponding K\"ahler metric: 
\begin{equation}
\nabla\mathcal{M}_{|u}:=-(R_{\omega_{u}}-\bar{R}),\label{eq:def of mab as gradient intro}
\end{equation}
 where $\bar{R}$ denotes the average scalar curvatures which, for
cohomology reasons, is a topological invariant. The geometric role
of the Mabuchi functional was elucidated by Donaldson \cite{do00}
who showed that - from a dual point of view - it can be identified
with the Kempf-Ness ``norm-functional'' for the natural action of
the group of all Hamiltonian diffeomorphisms on the space of all complex
structures on $X$ compatible with the symplectic form $\omega_{0}.$
This interpretation also provides a direct link between the Mabuchi
functional and the notion of stability in Geometric Invariant Theory
(GIT), which in the case when the K\"ahler class in question is integral,
i.e. equal to the first Chern class of an ample line bundle $L\rightarrow X,$
has been made precise in the seminal Yau-Tian-Donaldson conjecture
saying that $c_{1}(L)$ contains a K\"ahler metric with constant scalar
curvature if and only if the polarized manifold $(X,L)$ is K-stable
\cite{yau2,ti1,d0}.

\subsection{Statement of the main results}

As shown by Mabuchi \cite{mab0,mab-1} the functional $\mathcal{M}$
is convex along geodesics $u_{t}$ in the Riemannian manifold $\mathcal{H}(X,\omega).$
Unfortunately, given $u_{0}$ and $u_{1}$ in $\mathcal{H}$ there
may be no geodesic $u_{t}$ connecting them (see \cite{l-v,dar-l}
for recent counterexamples). Still by a result of Chen \cite{c0},
with complements due to Blocki \cite{bl}, there always exists a (unique)
\emph{weak }geodesic $u_{t}$ connecting $u_{0}$ and $u_{1}$ defined
as follows. First recall that, by an important observation of Semmes
\cite{se} and Donaldson \cite{do00}, after a complexification of
the variable $t,$ the geodesic equation for $u_{t}$ on $X\times[0,1]$
may be written as the following complex Monge-Amp\`ere equation on a
domain $M:=X\times D$ in $X\times\C$ for the function $U(x,t):=u_{t}(x):$
\begin{equation}
(\pi^{*}\omega+dd^{c}U)^{n\text{+1}}=0,\label{eq:ma eq for geod intro}
\end{equation}
As shown in \cite{c0,bl} for any smoothly bounded domain $D$ in
$\C$ the corresponding boundary value problem on $M$ admits a unique
solution $U$ such $\pi^{*}\omega+dd^{c}U$ is a positive current
with coefficients in $L^{\infty},$ satisfying the equation \ref{eq:ma eq for geod intro}
almost everywhere. In particular, when $D$ is an annulus in $\C$
this construction gives rise to the notion of a weak geodesic curve
$u_{t}$ in the extended space $\mathcal{H}_{1,1}$ of all functions
$u$ such that $\omega_{u}$ is a positive current with coefficients
in $L^{\infty}.$ Moreover, even if the original defining property
(formula \ref{eq:def of mab as gradient intro}) of the Mabuchi functional
requires that $\omega_{u}$ be positive and $C^{2}-$smooth (and in
particular that $u$ be $C^{4}-$smooth) Chen went on to show \cite{c1}
that the Mabuchi functional admits an explicit formula which is well-defined
along a weak geodesic ray $u_{t}$ as above. (This formula was also independently obtained by Tian, see \cite{ti2}.)  Indeed, 

\begin{equation}
\mathcal{M}(u)=\mathscr{E}(u)+\int_{X}\log(\frac{\omega_{u}^{n}}{\omega_{0}^{n}})\omega_{u}^{n},\label{eq:formula for mab intro}
\end{equation}
 where the first term $\mathscr{E}(u)$ is an explicit energy type
expression involving the integral over $X$ of a mixed Monge-Amp\`ere
expression of the form $u\omega_{u}^{j}\wedge\theta_{j}^{n-j}$ for
$j\in[1,n],$ where $\theta_{j}$ are explicit smooth forms depending
on $\omega_{0}$. The second term is the classical entropy of the
measure $\omega_{u}^{n}$ relative to the reference volume form $\omega_{0}^{n}.$
As a consequence $\mathcal{M}$ is naturally defined and finite on
the space $\mathcal{H}_{1,1},$ where the weak geodesics live. It
has been conjectured by Chen that $\mathcal{M}(\phi_{t})$ is convex
along any weak geodesic as above \cite{c1} (the case when $c_{1}(X)\leq0$
was settled by Chen). Our main result confirms this conjecture: 
\begin{thm}
\label{thm:main intro}For any K\"ahler class $[\omega]$ the Mabuchi
functional $\mathcal{M}$ is convex along the weak geodesic $u_{t}$
connecting any two points $u_{0}$ and $u_{1}$ in the space $\mathcal{H}$
of $\omega-$K\"ahler potentials.
\end{thm}
We will also show (Theorem \ref{thm:main text}) that $\mathcal{M}$
is 'weakly subharmonic' (see section 3 for precise definitions) subharmonic along any curve $u_{\tau}$ satisfying the complex
Monge-Amp\`ere equation \ref{eq:ma eq for geod intro} on $X\times D,$
as long as Chen's regularity property holds, i.e. $\pi^{*}\omega+dd^{c}U$
is a positive current with coefficients in $L^{\infty}.$ The subharmonicity
of the Mabuchi functional under stronger regularity assumptions on
the solution $U$ to the equation \ref{eq:ma eq for geod intro} (so
called ``almost smooth'' solutions) has been shown by Chen-Tian
\cite{c-t}. The key point of the proof of Theorem \ref{thm:main intro}
is a new local positivity property of the relative canonical line
bundle $K_{M/D}$ along the one-dimensiona
 current 
\[
S:=(\pi^{*}\omega+dd^{c}U)^{n}
\]
 in the product $M=X\times D.$ This can be seen as a generalization
of a positivity property of Monge-Amp\`ere foliations due to Beford-Burns
\cite{b-b}, further developed by Chen-Tian \cite{c-t}, since $S$
can be realized as an average of the leaves of such a foliation, when
it exists. But it should be stressed that one of the main points of
our approach is that it does not require the existence of any sort
of Monge-Amp\`ere foliation. Our proof uses plurisubharmonic variation of
local Bergman kernels (\cite{Mai-Yam}, \cite{bern00}); see Section \ref{sub:A-sketch-of} below for
a sketch of the proof and Section \ref{sub:A-positivity-property}
for comparison with previous results. 

We will also give some applications of Theorem \ref{thm:main intro}
to K\"ahler geometry, which have previously - in their full generality
- only been shown using the partial regularity theory of Chen-Tian
\cite{c-t}. Very recently however it has been showed by Julius Ross and David Witt Nystr\"om (see \cite{Ross-Witt}) that the partial regularity results do not hold as stated  in \cite{c-t}, so it seems that the earlier proofs are not complete.

We start with the following corollary which follows immediately
from the previous theorem, using the ``sub-slope property'' of convex
functions.
\begin{cor}
\label{cor:csck min mab etc intro}Any K\"ahler metric with constant
scalar curvature metric minimizes the corresponding Mabuchi functional.
More precisely, the following inequality holds 
\begin{equation}
\mathcal{M}(u_{1})-\mathcal{M}(u_{0})\geq-d(u_{1},u_{0})\sqrt{\mathcal{C}(u_{0})},\label{eq:chens ineq in intro}
\end{equation}
 for any two K\"ahler potentials $u_{0}$ and $u_{1}$ on a K\"ahler manifold
$(X,\omega),$ where $d$ is the distance function corresponding to
the Mabuchi metric and $\mathcal{C}$ denotes the Calabi energy, i.e.
$\mathcal{C}(u):=\int(R_{\omega_{u}}-\bar{R})^{2}\omega_{u}^{n}$
\end{cor}
The minimizing property above was first shown by Chen in the case
when the first Chern class $c_{1}(X)$ is non-positive and by Donaldson
\cite{d00,d00II}, in the case when the K\"ahler class in question is
integral, i.e. when it coincides with the first Chern class of an
ample line bundle $L$ over $X.$ The general case was treated by
Chen-Tian in \cite{c-t}, using their partial regularity theory and
approximation arguments and the inequality \ref{eq:chens ineq in intro}
was then obtained by Chen, building on \cite{c-t}. 

In the case of smooth geodesics it is well-known that the Mabuchi
functional $\mathcal{M}$ is \emph{strictly }convex modulo automorphisms,
or more precisely modulo the group $\mbox{Aut}_{0}(X)$ defined as
the connected component of the identity in the group of all biholomorphisms
of $X.$ If one could establish the corresponding strict convexity
for \emph{weak }geodesics - which seems very challenging - then it
would immediately imply the uniqueness modulo $\mbox{Aut}_{0}(X)$
of the critical points of $\mathcal{M}$, i.e. of cohomologous K\"ahler
metrics with constant scalar curvature. Here we will show that the
conjectural general strict convexity result referred to above is not
needed to establish the uniqueness result in question; it follows
from a rather general argument combining the convexity in Theorem
\ref{thm:main intro} with the well-known fact that the strict convexity
modulo $\mbox{Aut}_{0}(X)$ does hold at the linearized level (in
other words, the Hessian of $\mathcal{M}$ at a critical point of
$\mathcal{M}$ degenerates precisely along the action of holomorphic
vector fields). 
\begin{thm}
\label{thm:uniqueness intro} Given any two cohomologous K\"ahler metrics
$\omega_{0}$ and $\omega_{1}$ on $X$ with constant scalar curvature
there exists  an element g in the connected component $\mbox{Aut}_{0}(X)$
of the identity in the group of all biholomorphisms of $X$ such that
$\omega_{0}=g^{*}\omega_{1}.$ 
\end{thm}
In the case when $[\omega]=c_{1}(X)$ this result is due to Bando-Mabuchi
\cite{b-m} while the case $[\omega]=c_{1}(L)$ with $\mbox{Aut}_{0}(X)$
trivial was shown by Donaldson \cite{d00}, using approximation with
so called balanced metrics attached to high tensor powers of the line
bundle $L.$ The general uniqueness result appears in \cite{c-t}.

Our approach to the uniqueness theorem consists in adding a small strictly convex perturbation to the Mabuchi functional. The perturbed functional is then strictly convex so it can then have at most one critical point. In case  $\mbox{Aut}_{0}(X)$ is discrete, or equivalently there are no nontrivial holomorphic vector fields on $X$, it follows from the implicit function theorem that near any (smooth) critical point of the Mabuchi functional there is a critical point of such a perturbed functional, so the Mabuchi functional can also have at most one critical point. In the general case, when  $\mbox{Aut}_{0}(X)$ is nontrivial, critical points of $\mathcal{M}$ cannot in general be approximated by critical points of the perturbed functional. (Indeed, if this were possible we would get absolute uniqueness instead of uniqueness modulo automorphisms.) However, we prove that such approximation is possible if we first move the critical point by a suitable automorphism, and this permits us to prove uniqueness modulo automorphisms in the general case.
This is the principle of the proof, but in order to avoid technical complications  (that arise when there are nontrivial holomorphic vector fields) we will instead work with 'approximately critical points' so in the end we  avoid the actual use of the implicit function theorem.

More specifically, we will consider the setting of K\"ahler metrics with
constant $\alpha-$twisted scalar curvature, defined with respect
to a given ``twisting form'' $\alpha,$ i.e. a smooth closed non-negative
$(1,1)-$form on $X$ (see Section \ref{sub:The-twisted-setting}),
as well as Calabi's extremal metrics (Section \ref{sub:Calabi's-extremal-metrics}).
As shown in \cite{fi} the twisted setting appear naturally in the
case when $X$ is realized as the base of a fibration whose fibers
are equipped with constant scalar curvature metrics (then the role
of the twisting form $\alpha$ is played by the corresponding Weil-Peterson
metric on the base $X$ describing the variation of the complex structures
of the fibers); see also \cite{s-t} for relation to the K\"ahler-Ricci
flow on varieties of positive Kodaira dimension and \cite{sto} for
the relation to the algebro-geometric slope stability of Ross-Thomas.
Let us finally point out that Theorem 1.1 can also be extended to Tian-Zhu's modified K-energy functional \cite{t-z1b}, whose critical points are K\"ahler-Ricci solitons (details will appear elsewhere).

\subsubsection{Further extensions and applications }

One new feature of our method, further exploited in the companion
paper \cite{bch} by Lu and the first author, is that it also has bearings on the uniqueness
and regularity problem for very weak minimizers of the (twisted) Mabuchi
functional. The point is that, extending the results in \cite{bbegz}
concerning the case when $[\omega]=c_{1}(X)$, the Mabuchi functional,
as defined by formula \ref{eq:formula for mab intro}, can be extended
to the ``finite energy'' completion  $\mathcal{E}^{1}(X,\omega)$
of the space $\mathcal{H}(X,\omega)$ introduced by Guedj-Zeriahi
\cite{g-z1}, with good continuity/compactness properties. In particular,
the corresponding uniqueness result in the finite energy setting can
be used to study the convergence properties of a weak version of the
Calabi flow. To briefly explain this recall that the latter flow,
in its classical form, may be defined as the down-ward gradient flow
of the Mabuchi functional on the infinite dimensional Riemann manifold
$\mathcal{H}(X,\omega)$ equipped with the Mabuchi metric. Even if
the long-time existence of the classical Calabi flow is still open
it was shown by Streets \cite{st1} that a weak version of the Calabi
flow, dubbed the K-energy minimizing movement, is always well-defined
on the metric completion of the Mabuchi space $\mathcal{H}(X,\omega).$
Building on \cite{bbegz} and the very recent work  by
Darvas and Guedj, \cite{dar}, \cite{dar2} and \cite{Guedj},  we will show in \cite{bch} that the K-energy minimizing emanating
from a given potential $u_{0}$ in $\mathcal{H}(X,\omega),$ gives
rise to a curve of finite energy potentials in $\mathcal{E}^{1}(X,\omega)$
that we will call the\emph{ finite energy Calabi flow }with the property
that the corresponding positive currents $\omega_{t}$ have a top
intersection $\omega_{t}^{n}$ defining a measure on $X$ with finite
entropy and good convergence properties. More precisely, the following
convergence result hods:
\begin{thm}
\cite{bch} Let $[\omega]$ be a K\"ahler class on $X$ and $\alpha$
fixed smooth closed $(1,1)-$ form on $X.$ Assume that $[\omega]$
contains a K\"ahler metric with constant $\alpha-$twisted scalar curvature
$\omega_{\alpha}$ and that either $\alpha>0$ or $X$ admits no non-trivial
holomorphic vector fields and $[\omega]$ is proportional to $c_{1}(X).$
Then the finite energy twisted Calabi flow $\omega_{t}$ converges
in the weak sense of currents on $X$ towards $\omega_{\alpha},$
as $t\rightarrow\infty.$ More precisely, the measures $\omega_{t}^{n}$
converge in entropy towards the volume form $\omega_{\alpha}^{n}$
of $\omega_{\alpha}.$
\end{thm}
The relation to previous results is discussed in \cite{bch}. Some
further interactions between the Mabuchi functional and the notions
of finite energy and entropy are also studied in \cite{bch}. For
example, it is shown that the extended Mabuchi functional remains
convex along finite energy geodesics. Moreover, using finite energy
geodesics one can define a notion of ``weak Mabuchi geodesics''
in the space $\mathcal{P}(X)$ of all probability measures on a compact
K\"ahler manifold $X,$ such that the space of all probability measures
$\mu$ with finite entropy becomes geodesically closed and such that
the entropy functional defined with respect to a K\"ahler metric with
non-negative Ricci curvature becomes geodesically convex. As explained
in \cite{bch} the latter convexity property can be seen as the complex
version of a fundamental convexity property in the setting of optimal
transport theory.

\subsection{\label{sub:A-sketch-of}A sketch of the proof of Theorem \ref{thm:main intro}}

Let us sketch the proof of Theorem \ref{thm:main text} in the special
case when $\omega_{u_{t}}$ is continuous and strictly positive. The
starting point is the following essentially well-known formula for
the second order variation of the Mabuchi functional:
\begin{equation}
d_{t}d_{t}^{c}\mathcal{M}(u_{t})=\int_{X}T,\,\,\, T:=dd^{c}\Psi\wedge(\pi^{*}\omega+dd^{c}U)^{n},\,\,\,\,\Psi_{t}:=\log(\omega_{u_{t}}^{n}).\label{eq:form for dd of mab intro}
\end{equation}
Here $\Psi$ denotes the local weight of the metric on the relative
canonical line bundle $K_{M/D}\rightarrow M$ induced by the metrics
$\omega_{u_{t}}$ on $TX$ and $\int_{X}$ denotes the fiber-wise
integral, i.e. the natural map pushing forward a form on $M:=X\times D$
to a form on the base $D.$ (This formula follows from \ref{eq:formula for mab intro}, using that $dd^c$ commutes with push-forwards.) The proof proceeds by showing that the
integrand $T$ in formula \ref{eq:form for dd of mab intro} is a
non-negative top form on $M$ and in particular its push-forward to
$D$ is also non-negative, as desired. First observe that we can locally
write $\pi^{*}\omega+dd^{c}U=dd^{c}\Phi$ for a local plurisubharmonic
function $\Phi(t,z)=\phi_{t}(z),$ defined on the unit-ball in $\C^{n}.$
Accordingly, $\omega_{u_{t}}^{n}$ may be written as $(dd^{c}\phi_{t})^{n}$
locally on $X$ and by well-known convergence results for Bergman
kernels going back to H\"ormander, Bouche \cite{bo} and Tian \cite{t0},
the form $T$ can thus be locally realized as the weak limit, as $k\rightarrow\infty,$
of the forms $T_{k}$ defined by 
\[
T_{k}:=dd^{c}\log B_{k\phi_{t}}\wedge(dd^{c}\Phi)^{n},
\]
 where $B_{k\phi}:=K_{k\phi}e^{-k\phi}$ is the Bergman function (density
of states function) for the Hilbert space of all holomorphic functions
on the unit ball equipped with the standard $L^{2}-$norm weighted
by the factor $e^{-k\phi}.$ Finally, by the  results on plurisubharmonic variation of Bergman kernels in
\cite{bern00} the function $\log K_{k\phi_{t}}$ is plurisubharmonic
on $X\times D$ and hence 
\begin{equation}
dd^{c}\log B_{k\phi_{t}}=dd^{c}\log K_{k\phi_{t}}-kdd^{c}\Phi\geq0-kdd^{c}\Phi\label{eq:decomp of bergman intro}
\end{equation}
 Since the latter form vanishes when wedged with $(dd^{c}\Phi)^{n}$
(by the geodesic equation) this show that $T_{k}\geq0.$ Hence letting
$k\rightarrow\infty$ shows that $T\geq0$, which concludes the proof
of Theorem \ref{thm:main intro} under the simplifying assumption
that $\omega_{u_{t}}$ be continuous and strictly positive. The proof
in the general case involves a truncation procedure (to compensate
the lack of strict positivity of the measures $\omega_{u_{t}}^{n})$
and a generalization of the Bergman kernel asymptotics used above
to the case when the curvature form $dd^{c}\phi$ is merely in $L_{loc}^{\infty}.$ 

An intriguing aspect of our proof is that it relies on the individual
positivity properties of the two currents $dd^{c}\log K_{k\phi_{t}}$
and $-kdd^{c}\Phi$ appearing in the decomposition \ref{eq:decomp of bergman intro}
and these two currents diverge in the ``semi-classical'' limit $k\rightarrow\infty$
(contrary to their sum which converges to $dd^{c}\Psi).$ Hence, our
decomposition argument does not seem to have any direct analog for
the current $dd^{c}\Psi$ itself.

Finally we would like to thank S\'ebastien
 Boucksom and Mihai P\u aun
 for pointing out an omission in the first version of this paper regarding the continuity of the K-energy. After the first version of our paper was posted on the Arxiv, an alternative proof of the convexity of the K-energy, based on Monge-Amp\`ere equations instead of Bergman kernel has also been posted by XX Chen, L Li and M P\u aun
, see \cite{Chen-Li-Paun}. (In this paper it is also proved that $\M$ is subharmonic ( not just weakly subharmonic) along any complex $\mathcal{C}^{1,1}$ curve $u_\tau$ satisfying the complex Monge-Amp\`ere equation.)

\section{\label{sec:Regularity-of-weak}Weak geodesics and Bergman kernel
asymptotics}

\subsection{Preliminaries}

We start by introducing the notation for (quasi-) psh functions and
metrics on line bundles that we will use. Let $(X,\omega_{0})$ be
a compact complex manifold of dimension $n$ equipped with a fixed
K\"ahler form $\omega_{0},$ i.a. a smooth real positive closed $(1,1)-$form
on $X.$ Denote by $PSH(X,\omega_{0})$ the space of all $\omega_{0}-$psh
functions $u$ on $X,$ i.e. $u\in L^{1}(X)$ and $u$ is strongly
upper-semicontinuos (usc) and 
\[
\omega_{u}:=\omega_{0}+\frac{i}{2\pi}\partial\bar{\partial}u:=\omega_{0}+dd^{c}u\geq0,
\]
 holds in the sense of currents. We will write $\mathcal{H}(X,\omega_{0})$
for the interior of $PSH(X,\omega_{0})\cap\mathcal{C}^{\infty}(X),$
i.e. the space of all K\"ahler potentials (w.r.t $\omega_{0}).$ In
the\emph{ integral case}, i.e. when $[\omega]=c_{1}(L)$ for a holomorphic
line bundle $L\rightarrow X,$ the space $PSH(X,\omega_{0})$ may
be identified with the space $\mathcal{H}_{L}$ of (singular) Hermitian
metrics on $L$ with positive curvature current. We will use additive
notion for metrics on $L,$ i.e. we identify an Hermitian metric $\left\Vert \cdot\right\Vert $
on $L$ with its ``weight'' $\phi.$ Given a covering $(U_{i},s_{i})$
of $X$ with local trivializing sections $s_{i}$ of $L_{|U_{i}}$
the object $\phi$ is defined by the collection of open functions
$\phi_{|U_{i}}$ defined by 
\[
\left\Vert s_{i}\right\Vert ^{2}=e^{-\phi_{|U_{i}}}
\]
The (normalized) curvature $\omega$ of the metric $\left\Vert \cdot\right\Vert $
is the globally well-defined $(1,1)-$current defined by the following
local expression: 
\[
\omega=dd^{c}\phi
\]
The identification between $\mathcal{H}_{l}$ and $PSH(X,\omega_{0})$
referred to above is obtained by fixing $\phi_{0}$ and identifying
$\phi$ with the function $u:=\phi-\phi_{0},$ so that $dd^{c}\phi=\omega_{u}.$

\subsubsection{\label{sub:The-regularity-of} Weak geodesics and the space $\mathcal{H}_{1,1}$}

As recalled in the introduction of the paper equipping the space $\mathcal{H}(X,\omega_{0})$
with the Mabuchi's Riemannian metric a curve $u_{t}$ in $\mathcal{H}(X,\omega_{0})$
is a geodesic iff it satisfies a complex Monge-Amp\`ere equation. More
precisely, writing $t=\log|\tau|$ for $\tau\in\C$ so that $u_{t}$
may be identified with an $S^{1}-$invariant function $U$ on $M:=X\times D,$
where $D$ denotes the corresponding annulus in $\C,$ the $\pi^{*}\omega-$psh
function $U$ (with $\pi$ denoting the natural projection from $M$
to $X)$ satisfies 

\begin{equation}
(\pi^{*}\omega+dd^{c}U)^{n\text{+1}}=0,\label{eq:ma eq in section reg}
\end{equation}
where $U$ thus coincides at the boundary $\partial M$ with the function
determined by $u_{0}$ and $u_{1}.$ As shown in \cite{c0,bl} the
previous boundary value problem always admits (for any bounded domain
$D$ in in $\C$ a weak solution in the sense that $\pi^{*}\omega+dd^{c}U$
is a positive current with bounded coefficients, up to the boundary. We say that such functions have $\mathcal{C}^{1,1}_\C$-regularity. 
In particular any given two points $u_{0}$ and $u_{1}$ in $PSH(X,\omega_{0})$
are connected by a (unique) weak geodesic $u_{t}$ as above, defining
a curve in the space $\mathcal{H}_{1,1}\subset PSH(X,\omega_{0})$
of all $u$ such that $\omega+dd^{c}u$ is a positive current with
components in $L_{loc}^{\infty}.$

\subsection{\label{sub:Bergman-kernel-asymptotics}Bergman kernel asymptotics}

Given a (possibly non-compact) complex manifold $Y$ with a line bundle
$L\rightarrow Y$ equipped with a (bounded) metric $\phi$ we denote
by $K_{k\phi}$ the  section of $(kL+K_{Y})\otimes\overline{(kL+K_{Y})}\rightarrow Y$
determined by the restriction to the diagonal of the Bergman kernel
of the space $H^{0}(Y,kL+K_{Y})$ of all global holomorphic section
of $kL+K_{Y}$ (viewed as holomorphic $n-$forms on $Y$ with values
in $kL)$ equipped with the standard $L^{2}-$norm determined by the
metric $\phi$ (assumed to be finite): 
\begin{equation}
K_{k\phi}(x)=\sup_{s\in H^{0}(Y,kL+K_{Y})}\frac{s\wedge\bar{s}(x)}{\int_{Y}s\wedge\bar{s}e^{-k\phi}}\label{eq:def of bergman kernel}
\end{equation}
 In particular, contracting the corresponding metrics on $kL$ gives
a measure on $Y$ that, after a scaling, we write as
\begin{equation}
\beta_{k}:=\frac{n!}{k^{n}}K_{k\phi_{t}}e^{-k\phi_{t}}\label{eq:def of bergman meas}
\end{equation}
 By well-known Bergman kernel asymptotics (due to Bouche \cite{bo} and
Tian \cite{t0}, independently ) in the case when $Y=X$ the convergence
$\beta_{k}\rightarrow(dd^{c}\phi)^{n}$ holds as $k\rightarrow\infty,$
uniformly on $X$, if $\phi$ is $C^{2}-$smooth and strictly positively
curved, i.e. $dd^{c}\phi>0.$ However, in our setting $\phi$ will
only have a Laplacian in $L_{loc}^{\infty}$ (and not be strictly
positively curved), i.e. $\phi$ will be in $\mathcal{H}_{1,1}$ and
hence the convergence cannot be uniform in general. Moreover, unless
the given class $[\omega]$ on $X$ is integral there will be no line
bundle $L$ over $X$ and then we will have to let $Y$ be a small
coordinate ball, identified with the unit-ball in $\C^{n},$ taking
$L$ as a the trivial line bundle. In the next theorem we show that  a sufficiently strong
version of the convergence still holds in this setting. 
\begin{thm}
\label{Thm:bergmankernel as}Let $L\rightarrow Y$ be a line bundle
over a (possibly non-compact) complex manifold $Y$ and assume that
$L$ extends to a holomorphic line bundle over a compact complex manifold
$X$ equipped with a (singular) metric $\phi$ such that the curvature
current $dd^{c}\phi$ is non-negative with components in $L_{loc}^{\infty}$
(i.e. $\phi$ is in $\mathcal{H}_{1,1})$. Denote by $\beta_{k}$
the Bergman measure on $Y$ defined with respect to the restricted
metric on $Y.$ Then, given a smooth volume form $dV$ on a compact subdomain
$E$ of $Y$ there exists a constant $C$ such that 
\begin{equation}
\beta_{k}\leq CdV\label{eq:uniform upper bound on bergman measure in proof main}
\end{equation}
on $E,$ where the constant $C$ only depends on an upper bound on
the sup-norm of $dd^{c}\phi$ on $E.$ Moreover, $\beta_{k}(x)\rightarrow(dd^{c}\phi)^{n}$
in total variation norm on $E.$\end{thm}
\begin{proof}
\emph{Step one: upper bounds.} We will start by showing the uniform
upper bound \ref{eq:uniform upper bound on bergman measure in proof main}
together with the following point-wise upper bound: 
\begin{equation}
\limsup_{k\rightarrow\infty}\beta_{k}(x)\leq(dd^{c}\phi)^{n}(x)\label{eq:pt-wise upper bound on bergman meas}
\end{equation}
at almost any point $x$ of $Y$ (recall that by assumption the r.h.s
above has a density which is well-defined almost everywhere on $X,$
so this statement indeed makes sense). The proof will be completely
local. Given any point $x_{0}\in X$ and local holomorphic coordinates
$z$ centered at $x_{0}$ we take a local trivializing section $s$
of $L$ such that $\phi$ is represented by a function $\phi(z)$
satisfying $\phi(0)=0.$ Any given holomorphic section of $L$ may,
locally, be written as $f(z)s$ for a local holomorphic function $f.$
In particular, the function $\log|f|^{2}$ is subharmonic and hence
by the sub-mean inequality for subharmonic functions we have 
\[
\log|f|^{2}(0)\leq\int\log|f|^{2}d\sigma_{r},
\]
 where $d\sigma_{r}$ denotes the invariant probability meaure on
the sphere $|z|=r.$ Writing $\log|f|^{2}=\log(|f|^{2}e^{-k\phi})+k\phi$
in the r.h.s above and applying Jensen's inequality gives 
\[
|f|^{2}(0)\exp(-\int k\phi d\sigma_{r})\leq\int|f|^{2}e^{-k\phi}d\sigma_{r}
\]
Accordingly, multiplying both sides with $r^{2n-1},$ integrating
over $r\in[0,Rk^{-1/2}]$ and performing the change of variables $r\mapsto rk^{1/2}$gives
\begin{equation}
|f|^{2}(0)\left(\int_{|z|\leq Rk^{-1/2}}|f|^{2}e^{-k\phi}dV\right)^{-1}\leq C_{R,k}:=\left(\int_{0}^{R}e^{-r^{2}a_{\phi}(rk^{-1/2})}r^{2n-1}dr\right)^{-1},\label{eq:def of constant ckr}
\end{equation}
where
$$
\,\,\, a_{\phi}(r)=\frac{1}{r^{2}}\int_{|z|=r}\phi d\sigma_{r}.
$$
We claim that 
\begin{equation}
(i)\,\left|a_{\phi}(r)\right|\leq C,\,\,\,\,(ii)\,\lim_{r\rightarrow0}a_{\phi}(r)=a_{\phi}(0)=\frac{1}{n}(\Delta\phi)(0)\label{eq:claim in proof bergman}
\end{equation}
where $C$ only depends on an upper bound on $\Delta\phi$ on $B(r):=\{|z|\leq r\}$
and where $(ii)$ holds if $0$ is a Lesbegue point for $\Delta\phi$.
(Recall that $0$ is Lesbegue point for an $L^{1}-$function $h$
if 
$$h(0)=\lim_{r\rightarrow0}\frac{1}{V(B(r))}\int_{|z|\leq r}hdV,
$$
where $V$ denotes the volume of the ball $B(r)$.) Accepting this
claim for the moment we can first set $R=1$ and deduce from $(i)$
that $\beta_{k}(x)$ is uniformly bounded on any compact subset $E.$
Moreover, to get the precise pointwise bound \ref{eq:pt-wise upper bound on bergman meas}
we assume that $x$ is a Lesbegue point for the components of the
current $(dd^{c}\phi)(x),$ i.e. that $0$ is a Lesbegue point for
for the $L_{loc}^{\infty}-$functions representing the distributional
partial deriviatives $\frac{\partial^{2}\phi}{\partial z_{i}\partial\bar{z_{j}}}$
The complement of the set of all such points $x$ is a null set for
Lesbegue measure (as follows from Lebesgue\textquoteright{}s theorem).

Letting $k\rightarrow\infty$ and applying the dominated convergence
theorem for $R$ fixed  gives, by computing the Gaussian integral,
\[
\lim_{R\rightarrow\infty}\lim_{k\rightarrow\infty}C_{R,k}=\left(\int_{0}^{\infty}e^{-r^{2}a_{\phi}(0)}r^{2n-1}dr\right)^{-1}=\frac{(a_{\phi}(0))^{n}}{\pi{}^{n}}
\]
Now recall that  $a_{\phi}(0)=\frac{1}{n}(\Delta\phi)(0)$, so what we  need to do is to  replace the Laplacian, i e the trace of $dd^c\phi$, by the determinant of the same form. For this we note that we can make an arbitrary linear change of variables in the coordinates $z$ without changing the Bergman kernel estimate, if the determinant of the change of variables equals 1. First we  change coordinates so that the Hessian of $\phi$ is diagonal at the origin. Then we apply a diagonal change of coordinates $z_j\to \mu_j z_j$ with determinant one. By the arithmetic-geometric mean inequality, the infimum 
$$
\inf_{\mu_j}(1/n) \sum \lambda_j\mu_j
$$
over all positive $\mu_j$ with product 1 equals $(\Pi \lambda_j)^{1/n}$, so taking the infimum over all such changes of coordinates we get that
$$
\limsup_{k\to \infty}\beta_k\leq det(\phi_{j, \bar k}).
$$
This concludes the proof of Step one up to the proof of $(i)$
and $(ii)$ in \ref{eq:claim in proof bergman} to which we next turn.

First note that in order to establish $(ii)$ it will be enough to
show that the limit $a_{\phi}(0)$ exists and only depends on $(\Delta\phi)(0).$
Indeed, we can then replace $\phi(z)$ with $\phi_{0}(z)=|z|^{2}$
and note that, by symmetry, $a_{\phi_{0}}(0)=1=\frac{1}{n}(\Delta\phi_{0})(0).$
Denote by $g(z)$  the standard spherical symmetric fundamental
solution for the corresponding local Euclidean Laplacian $\Delta:=\sum_{i}\frac{\partial^{2}}{\partial z_{i}\partial\bar{z_{i}}}$
satisfying 
\begin{equation}
g(1)=0,\,\,\,\frac{\partial}{\partial r}g(r)=c_{n}\frac{1}{r^{2n-1}}
\end{equation}
Using Green's formula and integration by parts gives

\[
c_n a_{\phi}(R)=R^{-2}\int_{|z|\leq R}(\Delta\phi)gdV=R^{-2}\int_{0}^{R}A_{\phi}(r)\frac{\partial}{\partial r}g(r)dr\label{eq:deriv of green f}\]
 where 
\[
A_{\phi}(r):=\int_{|z|\leq r}\Delta\phi dV
\]
In particular, since $\Delta\phi\leq C$ on $B(r)$ this proves $(i)$
in \ref{eq:claim in proof bergman}. Moreovever, if $0$ is a Lesbegue
point for $\Delta\phi$ then we get $A_{\phi}(r)=V(B_{1})r^{2n}(\Delta\phi)(0)+o(r^{2n})$
and hence, using formula \ref{eq:deriv of green f}, 
\[
a_{\phi}(R)=V(B_{1})(\Delta\phi)(0)R^{-2}\int_{0}^{R}r(1+o(1))dr\rightarrow\frac{1}{2}c_{n}V(B_{1})(\Delta\phi)(0),
\]
 as $R\rightarrow\infty.$ This shows that the limit $a_{\phi}(0)$
exists and only depends on $(\Delta\phi)(0),$ which proves $(ii)$
in \ref{eq:claim in proof bergman}.

\emph{Step two: convergence in total variation norm. }First note that
by the uniform and pointwise bounds on $\beta_{k}$ established in
the previous steps it will in order to prove the convergance in total
variation norm be enough to show that, for any compact subdomain $E$
of $Y$ 
\begin{equation}
\liminf_{k\rightarrow\infty}\int_{E}\beta_{k}\geq\int_{E}(dd^{c}\phi)^{n}\label{eq:lower bound for bergman meas}
\end{equation}
Indeed, writing $\beta_k=f_k dV$ and $(dd^c\phi)^n=f dV$ we get
$$
\|\beta_k-(dd^c\phi)^n\|=\int |f_k-f|dV=\int (f-f_k)dV +2\int(f-f_k)_{-},
$$
with $(f-f_k)_- = -\min( f-f_k,0)$. The limsup of the first integral is less than or equal to zero by  \ref{eq:lower bound for bergman meas}, and the limsup of the second integral is less than or equal to zero by Fatou's lemma (cf Lemma 2.2 in \cite{ber1}).

 Next we note that it will be enough to
consider the case when $Y$ is compact. Indeed, by assumption $(L,\phi)$
extends to a compact complex manifold $X$ (with the same hypothesis
on $\phi$ as on $Y$) and it follows immediately from the definition
of Bergman measures that 
\[
\beta_{k}\geq\beta_{k,X}
\]
where the right hand side is the Bergman measure defined with respect
to $(X,L,\phi).$ Hence, once we have established that the bound \ref{eq:lower bound for bergman meas}
holds for $\beta_{k,X}$ it will automatically hold for $\beta_{k}.$
Moreover, in the compact case of $X$ it will be enough to establish
 the bound \ref{eq:lower bound for bergman meas} for $E=X.$ Indeed,
as pointed out above it implies the convergence in total variation
norm on $X$ which in turn implies the lower bound \ref{eq:lower bound for bergman meas}
on $E$ for $\beta_{k,X}$ and hence the same lower bound on $E$
for $\beta_{k}.$

Finally, to prove the lower bound \ref{eq:lower bound for bergman meas}
for $X$ compact we can exploit that $H^{0}(X,kL+K_{X})$ is finite
dimensional. Indeed, by the Hilbert-Samuel formula, $\dim H^{0}(X,kL+K_{X})=k^{n}\int c_{1}(L)^{n}/n!+o(k^{n})$.
Moreover, by basic properties of Bergman kernels for finite dimensional
Hilbert spaces $\int_{X}\beta_{k,X}=\frac{n!}{k^{n}}\dim H^{0}(X,kL+K_{X})$
and hence  
\[
\lim_{k\rightarrow\infty}\int_{X}\beta_{k,X}=\int_{X}(dd^{c}\phi)^{n},
\]
 which, as pointed out above, concludes
the proof of the general convergence. 
\end{proof}
For our purposes it will be enough to consider the case when $Y$
is a Euclidean ball in $\C^{n}:$ 
\begin{cor}
\label{cor:Bergman}Let $\phi$ be a plurisubharmonic function defined on the
neighbourhood of $B_{1}$  such that $\Delta\phi\in L_{loc}^{\infty}$
and denote by $\beta_{k}$ the Bergman measure for the Hilbert space
of all holomorphic functions $f$ on \textbf{$B_{1}$} equipped with
the weighted $L^{2}-$norm \textup{$\int_{B_{1}}|f|^{2}e^{-k\phi}dV,$}
where $dV$ denotes Lesbegue measure. Then $\beta_{k}\leq C_{E}dV$
for any given compactly included subdomain $E$ of $B_{1}$ and, after
perhaps passing to a subsequence, \textup{
\[
\lim_{k\rightarrow\infty}\beta_{k}(x)=(dd^{c}\phi)^{n}(x)
\]
}\textup{\emph{ for almost any $x$ in $B_{1}.$}}\end{cor}
\begin{proof}
Taking $L$ to be the trivial holomorphic line bundle on $Y:=B_{1}$
it will be enough to show the extension property required by the previous
theorem. By assumption $\phi$ is in $\mathcal{H}_{1,1}(B_{1+\epsilon})$
and up to changing $\phi$ by a harmless additive constant we may
assume that $\phi\geq\delta>0$ on $B_{1+\delta}.$ Hence for $C$
sufficently large $\psi_{C}:=\max\{\phi,C\log|z|^{2}\}$ coincides
with $\phi$ on a neighbourhood of the closed unit-ball $B_{1}$ and
with $C\log|z|^{2}$ on $B_{1+\epsilon/2}.$ Moreover, the same property
holds when the max is replaced by a suitable regularized max ensuring
that $\psi_{C}$ is also in $\mathcal{H}_{1,1}(B_{1+\epsilon}).$
Finally, for $C$ a given positive integer we note that any function
coinciding with $C\log|z|^{2}$ on the complement of a given ball
$B_{R}$ centered at $0$ in $\C^{n}$ extends, in the standard way,
to define a metric on the $m$ th tensor power $\mathcal{O}(m)\rightarrow\P^{n}$
of the hyperplane line bundle on complex projective space, which is
smooth and of non-negative curvature on the complement of $B_{R}.$
This gives the required extension and concludes the proof since $L^{1}-$convergence implies almost everywhere convergence, after
passing to a subsequence. ( This reduction of a problem for local plurisubharmonic functions to a problem for global metrics on a line bundle was probably first used by Siu in \cite{Siu-Morse}). 
\end{proof}

\section{\label{sec:Convexity-of-the mab along chen}Convexity of the Mabuchi
functional along weak geodesics}

In this section will prove our main result, stated as Theorem \ref{thm:main intro}
in the introduction, using the convergence results for local Bergman
kernels proved in the previous section. We start by introducing some
notation. If $\omega$ is a K\"ahler form on $X$ then it induces a
metric $\psi_{\omega}$ on the anti-canonical line bundle $-K_{X}:=\Lambda^{n}TX$
for which we will use the suggestive notation 
\[
\psi_{\omega}=-\log(\omega^{n})
\]
 i.e. given local holomorphic coordinates $\psi_{\omega}$ is represented
by $-\log(\omega^{n}/idz_{1}\wedge d\bar{z}_{1}\wedge\cdots).$ More
generally, given a measure $\mu,$ absolutely continuous w.r.t Lebesgue
measure, we write $\psi_{\mu_{0}}$ for the corresponding metric on
$-K_{X}$ which, symbolically means that 
\[
\mu=e^{-\psi_{\mu}}
\]
 By definition $\mbox{Ric \ensuremath{\omega}}$ is the curvature
form of the metric $\psi_{\omega},$ i.e. $\mbox{Ric \ensuremath{\omega}}=dd^{c}\psi_{\omega}.$
The \emph{Mabuchi functional} $\mathcal{M}$ \cite{mab0} is, with
our normalization, the functional on $\mathcal{H}:=\mathcal{H}(X,\omega)$
implicitly defined by 
\begin{equation}
d\mathcal{M}_{|u}=-n\mbox{Ric}(\omega_{u})\wedge\omega_{u}^{n-1}+\bar{R}\omega_{u}^{n},\,\,\,\,\,\,\bar{R}:=\frac{nc_{1}(X)\cdot[\omega]^{n-1}}{[\omega]^{n}},\label{eq:defining prop of mab in text}
\end{equation}
 where $d\mathcal{F}_{|u}$ denotes the differential at $\phi$ of
a given functional $\mathcal{F}$ on the $\mathcal{H},$ i.e. the
measure defined by the following property: for any $v\in C^{\infty}(X)$
\[
\left\langle d\mathcal{F}_{|u},v\right\rangle =\frac{d}{dt}_{|t=0}\mathcal{F}(u_{t}),
\]
 where $u_{t}$ is any smooth curve in $\mathcal{H}$ such that $\frac{d}{dt}_{|t=0}u_{t}=v$
(assuming that the measure $d\mathcal{F}_{|u}$ exists). Given a curve
$u_{t}$ in $\mathcal{H}$ we will identify it with a function $U$
on $X\times D,$ for $D$ an annulus in $\C$ (compare section \ref{sec:Regularity-of-weak}). 

The starting point of the proof of Theorem \ref{thm:main intro} is the
explicit formula for the Mabuchi functional in \cite{c1}, which 
has an ``energy part'' and an ``entropy part''. As there are many
different notations (and normalizations) for the energy type functionals
in question we start by introducing our notation. Given a metric $\phi$
as above we will write 
\begin{equation}
\mathcal{E}(u):=\int_{X}\sum_{j=0}^{n}u\omega_{u}^{n-j}\wedge\omega_{0}^{j}\label{eq:def of energy of phi}
\end{equation}
Similarly, given a closed $(1,1)-$form (or current) $T$ we
set

\begin{equation}
\mathcal{E}^{T}(u):=\int_{X}u\sum_{j=0}^{n-1}\omega_{u}^{n-j-1}\wedge\omega_{0}^{j}\wedge T\label{eq:def of T-energy of phi}
\end{equation}
A standard computation shows that the corresponding differentials
are given by: 
\begin{equation}
d\mathcal{E}_{|u}=(n+1)\omega_{u}^{n},\,\,\, d\mathcal{E}_{|\phi}^{T}=n\omega_{u}^{n-1}\wedge T.\label{eq:differentials of energy and T-en}
\end{equation}
Similarly, the second order variations are given by:
\begin{equation}
d_{\tau}d_{\tau}^{c}\mathcal{E}(u_{\tau})=\int_{X}(\pi^{*}\omega+dd^{c}U)^{n+1},\,\,\, d_{\tau}d_{\tau}^{c}\mathcal{E}^{T}(\phi_{\tau})=\int_{X}(\pi^{*}\omega+dd^{c}U)^{n}\wedge\pi^{*}T,\label{eq:second order differentials of energy and T-energy}
\end{equation}
where $\int_{X}$ denotes the fiber-wise integral, i.e. the push-forward
map induced by the natural projection $\pi$ from $X\times D$ to
$X.$ Finally, we recall that the \emph{entropy} of a measure $\mu$
relative to a reference measure $\mu_{0}$ is defined as follows if $\mu$ is absolutely continuous with respect to $\mu_0$:
\begin{equation}
H_{\mu_{0}}(\mu):=\int_{X}\log\left(\frac{d\mu}{d\mu_{0}}\right)d\mu\label{eq:def of entr in smooth case}
\end{equation}
There is a well known interpretation of the entropy functional as a Legendre transform that we will have use for at several occasions later on, see \cite{Legendre-ref}.
\begin{prop}If $\mu_0$ and $\mu$ are  probability measures on $X$ such that $\mu$ is absolutely continuous with respect to $\mu_0$,  then
$$
H_{\mu_{0}}(\mu)=\sup_f \int_X fd\mu -\log\int_X e^f d\mu_0,
$$
where the supremum is taken over all continuous functions on $X$.
\end{prop}
\begin{proof} 
First note that Jensen's inequality gives
$$
\exp\int_X (f-\log(d\mu/d\mu_0))d\mu\leq\int_X e^fd\mu_0.
$$
Taking logarithms and rearranging this gives the $\geq$ direction of the inequality. The other direction follows by approximating $\log(d\mu/d\mu_0)$ by continuous functions $f$.
\end{proof}
For future use we record two immediate consequences of this: The entropy is a convex function of the measure $\mu$ for the natural affine structure on the space of probability measures. Second, as the supremum of a set of continuous functions, the entropy is lower semicontinuous with respect to the weak*-topology.

Now we can state the explicit formula in \cite{c1}, written in our
notation, for the Mabuchi functional $\mathcal{M}$ implicitly defined
(up to an additive constant) by formula \ref{eq:defining prop of mab in text}.
\begin{prop}
\label{prop:form for mab}Given a K\"ahler metric $\omega_{0}$ on $X$
with volume form $\mu_{0}:=\omega_{0}^{n}$ of total mass $[\omega]^{n}$
the following formula holds for the Mabuchi functional on the corresponding
space $\mathcal{H}$ of all K\"ahler potentials: 
\begin{equation}
\mathcal{M}(u)=\left(\frac{\bar{R}}{n+1}\mathcal{E}(u)-\mathcal{E}^{\mbox{Ric}\ensuremath{\omega_{0}}}(u)\right)+H_{\mu_{0}}(\omega_{u}^{n}),\,\,\,\,\,\bar{R}:=\frac{nc_{1}(X)\cdot[\omega_{0}]^{n-1}}{[\omega_{0}]^{n}}\label{eq:formula for mab}
\end{equation}
\end{prop}
\begin{proof}
For completeness and as a way to check our normalizations we recall
the proof. A direct calculation gives 
\[
\frac{d}{dt}H_{\mu_{0}}(\omega_{u_{t}}^{n})=0+\int\log\frac{\omega_{u_{t}}^{n}}{\omega_{0}^{n}}\frac{d\omega_{u_{t}}^{n}}{dt}=-n\int_{X}\frac{du_{t}}{dt}\mbox{Ric}\omega_{u_{t}}\wedge\omega_{u_{t}}^{n-1}+n\int_{X}\frac{du_{t}}{dt}\mbox{Ric}\omega_{0}\wedge\omega_{u_{t}}^{n-1}
\]
(using, in the first equality, that $\omega_{0}^{n}$ has the same
mass as $\omega_{u_{t}}^{n}$ and, in the second equality, one integration
by parts). Hence, since $d\mathcal{E}_{|u}^{T}=nT\wedge\omega_{u}^{n-1}$
(formula \ref{eq:differentials of energy and T-en}) we get $d(H_{\mu_{0}}-\mathcal{E}^{\mbox{Ric}\ensuremath{\omega_{0}}})=-n\mbox{Ric}\omega_{u}\wedge\omega_{u}^{n-1},$
which coincides with the first term in the defining expression for
$d\mathcal{M}_{|u}$ (formula \ref{eq:defining prop of mab in text}).
Finally, since $d\mathcal{E}_{|u}=(n+1)\omega_{u}^{n}$ (formula \ref{eq:differentials of energy and T-en})
this shows that the differential of the functional defined by the
r.h.s in formula \ref{eq:formula for mab} has the desired property.
\end{proof}
Following Chen \cite{c1} we now extend the functional $\mathcal{M}$
from $\mathcal{H}$ to the space $\mathcal{H}_{1,1}$ of all $u$
such that $\omega+dd^{c}u$ is a positive current with $L^{\infty}-$coefficients,
using the formula in the previous proposition. Theorem \ref{thm:main intro} claims that this functional is convex along weak geodesics.

 It is not a priori clear that the functional is continuous along weak geodesics. (We thank S\'ebastien Boucksom and  Mihai P\u aun
  for pointing this out to us.) It does follow from pluripotential theory that the energy parts of the formula  are continuous since the potential varies continuously from fiber to fiber. The entropy part however is only known to be lower semicontinuous.  Therefore we will first state the basic result concerning distributional derivatives, and then show  the required continuity in our setting afterwards. In the theorem below we say that a function $v$  of one complex variable is {\it weakly subharmonic} if $\partial\bar\partial v\geq 0$ in the sense of currents. Similarily, we say that a function of one real variable is {\it weakly convex} if its second derivative in the sense of distributions is nonnegative. 
\begin{thm}
\label{thm:main text} Let $u_{\tau}$ be a family of functions in
$PSH(X,\omega)$ such that $\omega+dd^{c}u$ is a locally bounded
current, \textup{$\pi^{*}\omega+dd^{c}U\geq0$ and $(\pi^{*}\omega+dd^{c}U)^{n+1}=0$
on $X\times D.$} Then the Mabuchi functional $\mathcal{M}(u_{\tau})$
is weakly subharmonic with respect to $\tau\in D.$\textup{\emph{ In particular,
$\mathcal{M}(u_{t})$ is weakly convex along the weak geodesic $u_{t}$ connecting
any two given points in $\mathcal{H}(X,\omega).$}}\end{thm}
\begin{proof}
Let $\Psi=\Psi(\tau,x)=\psi_\tau(x)$ be a locally bounded singular metric on the relative canonical
line bundle $K_{M/D}$ and denote by $f^{\Psi}(\tau)$ the following
function on $D$ attached to $\Psi:$ 
\[
f^{\Psi}(\tau):=\left(\frac{\bar{R}}{n+1}\mathcal{E}(u_{\tau})-\mathcal{E}^{\mbox{Ric}\ensuremath{\omega_{0}}}(u_{\tau})\right)+\int_{X}\log(\frac{e^{\psi_{\tau}}}{\omega_{0}^{n}})\omega_{u_{\tau}}^{n}
\]
(the definition is made so that $f^{\Psi}(\tau)=\mathcal{M}(u_{\tau})$
if $\Psi$ is the (unbounded) metric defined by $\omega_{u_{\tau}}^{n}).$
Then we claim that 
\begin{equation}
dd^{c}f^{\Psi}(\tau)=\int_{X}T,\,\,\,\,\, T:=dd^{c}(\Psi\wedge\left(\pi^{*}\omega_0+dd^{c}U\right)^{n})\label{eq:push-forward formula in proof thm main text}
\end{equation}
 where $T$ is defined as an $(n+1,n+1)$ current (distribution),
which a priori may not be of order zero. More precisely, for a local
smooth test function $v$ supported on a local coordinate neighborhood
$V\subset M$ the current $T$ is locally defined by
\[
\left\langle T,v\right\rangle =\int\Psi_{V}\left(\pi^{*}\omega_0+dd^{c}U\right)^{n}\wedge dd^{c}v,
\]
 where $\Psi_{V}$ is a local function representing the metric $\Psi$
on $K_{M/D}$ (given a local trivialization of $K_{M/D}).$ To prove
formula \ref{eq:push-forward formula in proof thm main text} take
a sequence $\Psi_{j}$ of uniformly bounded smooth metrics such that
$\Psi_{j}\rightarrow\Psi$ almost everywhere on $X$ for every $\tau$ (which may be
constructed using local convolution and a partition of the unity).
Then a direct calculation (using formula \ref{eq:second order differentials of energy and T-energy})
gives 
\begin{equation}
dd^{c}f^{\Psi_{j}}(\tau)=\eta_{j}:=\int_{X}T_{j},\,\,\,\, T_{j}:=dd^{c}(\Psi_{j}\wedge\left(\pi^{*}\omega_0+dd^{c}U\right)^{n})\label{eq:proof of push-forward form in proof}
\end{equation}
 By the dominated convergence theorem $\eta_{j}\rightarrow\eta:=\int_{X}T$
weakly on $D$ (in the sense of distributions). Moreover, by the dominated
convergence theorem $f^{\Psi_{j}}(\tau)\rightarrow f^{\Psi}(\tau)$
pointwise on $D,$ in a dominated manner and hence, since the linear
operator $dd^{c}$ is continuous under such convergence the desired
formula \ref{eq:push-forward formula in proof thm main text} follows
from formula \ref{eq:proof of push-forward form in proof}.

We want to apply these considerations to $\Psi=\log(\omega_0 +dd^c_XU)^n$, but we cannot do so immediately since this metric is not locally bounded. For this reason we next introduce a truncation in the following way.
For a fixed positive number $A,$ we define 
\[
\Psi_{A}:=\max\{\log\left(\omega_0+dd^{c}_{X}u_{\tau}\right)^n,\chi-A\}
\]
 where $\chi$ denotes a suitable fixed continuous metric on $K_{M/D},$
to be constructed below. We claim that the current
$$
T_{A}:=dd^{c}\Psi_{A}\wedge\left(\pi^{*}\omega_0+dd^{c}U\right)^{n}
$$
satisfies 
$T_{A}\geq0,$ i.e. is defined by a positive measure, if $\chi$ is chosen to be continuous and such that
\[
dd^{c}\chi\geq-k_{0}(\pi^{*}\omega_0+dd^{c}U)
\]
 for some positive integer $k_{0}$.
As explained above this will imply that 
\[
f^{\Psi_{A}}(\tau):=\left(\frac{\bar{R}}{n+1}\mathcal{E}(u_{\tau})-\mathcal{E}^{\mbox{Ric}\ensuremath{\omega_{0}}}(u_{\tau})\right)+\int_{X}\log(\max\left\{ \frac{\omega_{u}^{n}}{\omega_{0}^{n}},\frac{\chi-A}{\omega_{0}^{n}}\right\} \omega_{u_{\tau}}^{n}
\]
is subharmonic for any $A>0$. Letting $A\rightarrow\infty$
and invoking the dominated convergence theorem we get $f^{\Psi_{A}}(\tau)\rightarrow\mathcal{M}(u_{\tau})$
which will conclude the proof of the theorem.

To construct $\chi$ we first let
$\chi_{0}$ be an arbitrary smooth metric on $K_{X}$. Then  we  set $\chi:=\pi^{*}\chi_{0}-k_{0}U$
where $k_{0}$ is sufficiently large to ensure that $dd^{c}\chi_{0}+k_{0}\omega_{0}\geq0$. Then
$$
dd^c\chi= \pi^*dd^c \chi_0 -k_0(\pi^*\omega_0+dd^c U) +k_0 \pi^*\omega_0\geq
 -k_0(\pi^*\omega_0+dd^c U),
$$
so $\chi$ fulfills our requirement.

Now, the claim that $T_{A}\geq0$ is a local statement. Accordingly,
we locally write 
\[
\pi^{*}\omega_0+dd^{c}U=dd^{c}\Phi
\]
for a local psh function $\Phi$ on $M$ and write $\phi_{\tau}=\Phi(\cdot,\tau).$
Our proof proceeds by a local approximation argument involving the
local Bergman measures $\beta_{k\phi_{\tau}}$ (that we identify with
their density) for the Hilbert space of all holomorphic functions
on the unit-ball in Euclidean $\C^{n}$ equipped with the weight $k\phi_{\tau};$
see Section \ref{sub:Bergman-kernel-asymptotics}. More precisely,
consider the following local current: 
\[
T_{A,k}:=dd^{c}\Psi_{A,k}\wedge(dd^{c}\Phi)^{n},\,\,\,\,\,\Psi_{A,k}:=\max\{\log\beta_{k},\chi-A\}
\]
 By Prop \ref{Thm:bergmankernel as} and the dominated convergence
theorem 
\[
\lim_{k\rightarrow\infty}T_{k,A}=T_{A}
\]
in the local weak topology of currents. Thus, to prove that $T_{A}\geq0$
it will be enough to prove that the locally defined $(n+1,n+1)-$current
$T_{k,A}$ is a positive measure. To fix ideas we first observe that
the following current is positive: 
\[
T_{k}:=dd^{c}\Psi_{k}\wedge(dd^{c}\Phi)^{n},\,\,\,\,\,\Psi_{k}:=\log(\beta_{k})
\]
(which formally corresponds to the case $A=\infty).$ Indeed, by the
 results on plurisubharmonic variation of Bergman kernels in \cite{bern00} $dd^{c}\log K_{k\phi_{t}}\geq0$
on $X\times A$ and hence 
\begin{equation}
dd^{c}\log\beta_{k}\geq-kdd^{c}\Phi\label{eq:lower bound on ddc of log of bergm measure}
\end{equation}
As a consequence, 
\[
T_{k}:=dd^{c}\log\beta_{k}\wedge(dd^{c}\Phi)^{n}\geq-k(dd^{c}\Phi)\wedge(dd^{c}\Phi)^{n}=0,
\]
using the geodesic equation \ref{eq:ma eq in section reg} in the
last equality. Moving to the case when $A\neq\infty$ we note that,
by construction, $\Psi_{A,k}$ is the max of two local functions whose
curvature forms are bounded from below by $-kdd^{c}\Phi$ (for $k\geq k_{0})$
and hence $\Psi_{A,k}$ also satisfies 
\begin{equation}
dd^{c}\Psi_{A,k}\geq-kdd^{c}\Phi\label{eq:lower bound on ddc of truncated metric}
\end{equation}
Now arguing precisely as above (and using the inequality \ref{eq:lower bound on ddc of truncated metric})
we see that $T_{k,A}\geq0.$ Moreover, by Corollary \ref{cor:Bergman}
\[
e^{\Psi_{A,k}}:=\max\{\frac{n!}{k^{n}}K_{k\phi_{t}}e^{-k\phi_{t}},e^{-(\chi-A)}\}\rightarrow\max\{MA(\phi),e^{-(\chi-A)}\},
\]
 as $k\rightarrow\infty$ pointwise almost everywhere on $X$ and for every $\tau$ in a dominated
fashion (after passing to a subsequence with respect to $k).$ Hence,
invoking the dominated convergence theorem gives the following local
weak convergence: 
\[
\lim_{k\rightarrow\infty}T_{k,A}=T_{A}
\]
In particular, this shows that $T_{A}\geq0$ and as explained above
this concludes the proof of the theorem.
\end{proof}
Before going on to prove the continuity of the Mabuchi functional we point out that the previous proof simplifies somewhat in case the cohomology class of $\omega$ is integral. Then we can write 
$$
\omega_0 +dd^c u_\tau=dd^c\phi_\tau,
$$
where $\phi_\tau$ is for each $\tau$ the weight of a metric on a positive line bundle $L$. We can then consider the Bergman kernels for the spaces $H^0(X, K_X+kL)$, induced by the metrics $k\phi_\tau$, (instead of the local Bergman kernels that we used in the proof for the general case) and their Bergman measures 
$$
\beta_{k \tau}= K_{k\phi_\tau} e^{-k\phi_\tau}k^{-n}.
$$
We define 
$$
\Psi_A=\max\{\log(dd^c\phi_\tau^n), \chi -A\}
$$
and 
$$
\Psi_{A,k}=\max\{\log(\beta_{k \tau}), \chi-A\}
$$
and use  these metrics on the relative canonical bundle $K_{M/D}$ to define  functions $f^{\Psi_A}(\tau)$   and $f^{\Psi_{A, k}}(\tau)$ as in the very beginning of the proof. We then get that  pointwise $f^{\Psi_{A, k}}$ tends  to  $f^{\Psi_A}(\tau)$ as $k\to \infty$ and that  $f^{\Psi_A}(\tau)$ tends to $\mathcal{M}(\tau)$ as $A\to \infty$ . Moreover  $f^{\Psi_{A, k}}(\tau)$ is subharmonic by the same argument as before and it follows that $\mathcal{M}$ is at least weakly subharmonic. We will have use for this remark in the proof of the continuity.

\begin{thm}
\label{thm:continuity text} $\mathcal{M}$ is continuous along weak geodesics and therefore convex in the pointwise sense.
\end{thm}
\begin{proof}
Here we assume that the function $U$ defines a weak geodesic so we may assume that it depends only on $t:=\Re\tau$.
We  first consider the case when the class is integral. The functionals
 $f^{\Psi_{A, k}}(\tau)$ are then clearly continuous with respect to $\tau$ since by the continuity of the metric $\phi_\tau$, the Bergman kernels depend continuously on $\tau$. Hence  $f^{\Psi_{A, k}}$ are convex in the ordinary pointwise sense. These functions converge pointwise to  $f^{\Psi_{A}}$ as $k\to \infty$, so these functions are also convex. Finally, as $A\to \infty$ we get that the Mabuchi functional is also convex. As a convex function, $\mathcal{M}$ is thus continuous on the open interval and upper semicontinuous on the closed interval. By the lower semicontinuity of the entropy,  $\mathcal{M}$ is always lower semicontinuous, so we conclude that $\mathcal{M}$ is in fact continuous on the closed interval. 

We will now sketch how this argument can be adapted to the general case.  Then we define $\Psi_A$ as in the proof of Theorem 3.2. It is enough to prove that the corresponding function $f^{\Psi_A}$ is convex (in the pointwise sense) since then we can take the limit as $A\to -\infty$ and get that $\mathcal{M}$ is convex, and we conclude as in the integral case that $\mathcal{M} $ is continuous on the closed interval. 

Let $\kappa_\epsilon(s)$ be a sequence of strictly convex functions with $\kappa_\epsilon'\geq 1$  on the real line tending to $s$ as $\epsilon\to 0$. We define $f^{\Psi_A}_\epsilon$ just like $f^{\Psi_A}$, but replacing the factor 
$$
\log(\frac{e^{\psi_{A \tau}}}{\omega_0^n})
$$
in the entropy term by
$$
\kappa_\epsilon(\log(\frac{e^{\psi_{A \tau}}}{\omega_0^n})). 
$$
It is enough to prove that these functions are convex for all $\epsilon>0$ and we already know by the same argument as in the proof of Theorem 3.2 that they are weakly convex. We let $\xi_j^2$ be a partition of unity subordinate to a covering of coordinate patches over which $L$ is trivial and consider the local entropy functions
$$
H_j=\int_X \xi_j^2 \kappa_\epsilon(\log(\frac{e^{\psi_{A \tau}}}{\omega_0^n})).
$$
We define $H_j^{(k)}$ in a similar way, replacing $\Psi_A$ by its $k:th$ approximation by local Bergman kernels.
Taking the $dd^c$ of $H_j^{(k)}$,  using the plurisubharmonic variation of Bergman kernels and the strict convexity of $\kappa_\epsilon$ a direct estimate shows  that 
$$
dd^c H_j^{(k)}\geq - C_\epsilon,
$$
so $H_j^{(k)} + C_\epsilon t^2$ is convex since our local Bergman kernels depend continuously on $t$. Letting $k\to \infty$ we find that $H_j + C_\epsilon t^2$ is also convex. We can then sum over $j$ and conclude that 
$$
\int_X  \kappa_\epsilon(\log(\frac{e^{\psi_{A \tau}}}{\omega_0^n}))+ C_\epsilon t^2
$$
is convex and in particular continuous. Therefore  $f^{\Psi_A}_\epsilon$ are convex in the pointwise sense, since we already know that they are weakly convex. This completes the proof. 
\end{proof}

\subsection{Proof of Corollary \ref{cor:csck min mab etc intro}}

Fix $u_{0}$ and $u_{1}$ in $\mathcal{H}$ and denote by $u_{t}$
the corresponding weak geodesic. By the ``sub-slope inequality''
for the convex function $f(t):=\mathcal{M}(u_{t}),$ i.e. $f(1)-f(0)\geq f'(0)$
we have 
\[
\mathcal{M}(u_{1})-\mathcal{M}(u_{0})\geq f'(0)\geq\int_{X}(-R_{\omega_{u_{0}}}+\bar{R})\frac{du_{t}}{dt}_{|t=0}\omega_{u_{0}}^{n},
\]
 where the lower bound for $f'(0)$ is obtained by direct differentiations
as in the proof of Prop \ref{prop:form for mab} (see Lemma \ref{lem:deriv av mab along singular}
below). In particular, if $\omega_{u_{0}}$ has constant scalar curvature
then it minimizes the Mabuchi functional. More generally, applying
the Cauchy-Schwartz inequality to the right hand side of the inequality
above and using that $d(u_{0},u_{1})^{2}=\int(\dot {u_t}|_{t=0})^{2}\omega_{u_{0}}^{n}$
(see \cite{c0}) concludes the proof. 
\begin{lem}
\label{lem:deriv av mab along singular}Given and $u_{0},u_{1}\in\mathcal{H},$
let $u_{t}$ be the corresponding weak geodesic curve. Then 

\[
\lim_{t\rightarrow0^{+}}\frac{\mathcal{M}(u_{t})-\mathcal{M}(u_{0})}{t}\geq\int_{X}(-R_{\omega_{u_{0}}}+\bar{R})\frac{du_{t}}{dt}_{|t^{+}=0}\omega_{u_{0}}^{n}
\]
\end{lem}
\begin{proof}
This is shown by refining the argument in the proof of Prop \ref{prop:form for mab}. We will first handle the entropy part, i e show that
$$
\lim_{t\to 0}(1/t)(H_{\mu_0}(\omega_{u_t}^n)-H_{\mu_0}(\omega_{u_0}^n))
\geq
$$
$$
-n\int_{X}\frac{du_{t}}{dt}_{t=0}\mbox{Ric}\omega_{u_{0}}\wedge\omega_{u_{0}}^{n-1}+n\int_{X}\frac{du_{t}}{dt}|_{t=0}\mbox{Ric}\omega_{0}\wedge\omega_{u_{0}}^{n-1}.
$$
Here we use the fact that the entropy is convex with respect to the affine structure on the space of probability measures (cf Proposition 3.1) , so that
$$
H_{\mu_0}(\nu_1)-H_{\mu_0}(\nu_0)\geq (d/ds)|_{s=0}H_{\mu_0}(\nu_s)
$$
if $\nu_s=s\nu_1+(1-s)\nu_0$. Moreover, since $\log(\nu_s/\mu_0)\nu_s$ is convex in $s$, it follows from monotone convergence that
$$
(d/ds)|_{s=0}H_{\mu_0}(\nu_s)=\int \log(\nu_0/\mu_0)(d\nu_1-d\nu_0).
$$
From this we get, choosing $\nu_1=\omega_{u_t}^n$ and $\nu_0=\omega_{u_0}^n$ that
\[
\frac{1}{t}\left(H_{\mu_{0}}(\omega_{u_{t}}^{n})-H_{\mu_{0}}(\omega_{u_{0}}^{n})\right)\geq \int\log(\omega_{u_{0}}^n/\mu_{0})\frac{1}{t}\left(\omega_{u_{t}}^{n}-\omega_{u_{0}}^{n}\right)
\]
Expand $\omega_{u_{t}}^{n}-\omega_{u_{0}}^{n}=dd^{c}(u-u_{t})\wedge(\omega_{u_{0}}^{n-1}+...\omega_{u_{t}}^{n-1})$
and use integration by parts to let the $dd^{c}-$operator instead
act on the smooth function $\log\frac{\omega_{u_{0}}^{n}}{\mu_{0}}$.
Then letting  $t\rightarrow0$ we get  the desired inequality for
the entropy part of $\mathcal{M}(u_{t}).$ The calculation
for the derivative of the ``energy part'' of $\mathcal{M}$ follows
immediately from the relations \ref{eq:differentials of energy and T-en}.
\end{proof}

\subsubsection{\label{sub:The-twisted-setting}The twisted setting }

Later on we will also consider 'twisted'  versions of the Mabuchi functional. These are obtained simply as the sum of $\mathcal{M}$ and another convex functional $\mathcal{F}$. We will consider two main cases. The first is to let $\mu$ be a strictly positive smooth volume form on $X$ and put
$$
\mathcal{F}(u)=\mathcal{F}_\mu(u):= \int_X u d\mu -c_\mu\mathcal{E}(u),
$$
with $c_\mu$ choosen so that $\mathcal{F}_\mu(1)=0$.
Clearly $\F_\mu$ is convex along weak geodesics since its derivative is
$$
(d/dt)\F_\mu(u_t)=\int_X u'_t d\mu -(d/dt)\e(u_t).
$$
The first term here is increasing since $u'_t$ is increasing, and the second term is constant since the energy is linear along weak geodesics. 
The next choice is to let $\alpha$ be a strictly positive $(1,1)$-form on $X$
and let
$$
\mathcal{F}=\mathcal{F}_\alpha:= \mathcal{E}^\alpha -c_\alpha\mathcal{E},
$$
the constant $c_\alpha$ again chosen so that $\mathcal{F}$ vanishes on constants. By formula (3.5) $\mathcal{F}_\alpha$ is again convex along (sub)geodesics, since it is clearly continuous. (The strict convexity seems to be a more subtle issue that for simplicity we do not discuss here.) The critical points of $\mathcal{M}_\alpha:= \mathcal{M}+\mathcal{F}_\alpha$ are said to have constant  $\alpha$-twisted scalar curvature, i e  they satsify an equation
$$
R_\omega -\mbox{tr}_\omega(\alpha) = \mbox{constant}_\alpha,
$$
see \cite{fi}, \cite{sto}.
Just as before it follows that any metric with constant $\alpha$-twisted scalar curvature minimizes $\mathcal{M}_\alpha$.
As a consequence,
the $\alpha-$twisted Mabuchi functional is bounded from below in
any K\"ahler class containing a metric with constant $\alpha-$twisted
scalar curvature. As shown in \cite{sto} this leads to geometric
obstructions for the existence of such metrics.

\subsection{\label{sub:A-positivity-property}A positivity property for solutions
to homogeneous Monge-Amp\`ere equation and its relation to foliations}

The proof of Theorem \ref{thm:main text} yields the following positivity
result of independent interest, for sufficiently regular solutions
to the local homogeneous Monge-Amp\`ere equation on a product domain
(in the proof of Theorem \ref{thm:main text} the role of the current
$S$ below is played by $(dd^{c}\Phi)^{n}$): 
\begin{thm}
\label{thm:pos prop}Let $\Phi$ be a plurisubharmonic function on
$M:=X\times D$ where $X$ and $D$ are domains in $\C^{n}$ and $\C,$
respectively and assume that the positive current $dd^{c}\Phi$ has
components in $L_{loc}^{\infty}$ and satisfies $(dd^{c}\Phi)^{n+1}=0.$
Then the singular metric induced by the fiberwise currents $\omega_{\tau}:=dd^{c}\phi_{\tau}$
on the relative canonical line bundle $K_{M/D}\rightarrow M$ has
non-negative curvature along any positive current $S$ in $M$ of
bidimension $(1,1)$ with the property that $\Phi$ is harmonic along
$S,$ i.e. $\left\langle dd^{c}\Phi,S\right\rangle =0.$ More precisely,
for any positive number $A$ 

\[
i\partial\bar{\partial}\log_{A}\det(\frac{\partial^{2}\phi_{\tau}}{\partial z_{i}\partial\bar{z}_{j}})\wedge S\geq0,
\]
 in terms of the truncated logarithm defined by $\log_{A}t:=\max\{\log t,-A\}.$ 
\end{thm}
In particular, if $\Phi$ happens to admit a Monge-Amp\`ere foliation
then the positivity result above holds along the leaves of the foliation.
This observation is closely related to a previous local result of
Bedford-Burns (see Prop 4.1 in \cite{b-b}) and Chen-Tian who considered
the case when $\Phi$ corresponds to a global bona fide geodesic $u_{t}$
in the space of K\"ahler potentials on a K\"ahler manifold $(X,\omega)$
(see (Corollary 4.2.11 in \cite{c-t}). Then, by a classical result
of Bedford-Kalka (which only demands that $\Phi$ be $C^{3}-$smooth),
there is a foliation of $M:=X\times D$ in one-dimensional complex
curves $\mathcal{L}_{\alpha}$ (the leaves) such that the local potential
$\Phi$ is harmonic along any leaf $\mathcal{L}_{\alpha}.$ Moreover,
the leaves are transverse to the slice $X\times\{0\}$ (and hence
the latter space can be used as the parameter space for the set of
leaves). In this setting the results of  Bedford-Burns and Chen-Tian referred
to above may be formulated as the following special case of the previous
theorem:
\begin{prop}
\label{prop:pos along leaves}Consider the relative canonical line
bundle $K_{M/A}$ with the smooth metric induced by the volume forms
$(dd^{c}\phi_{t})^{n}.$ Then its restriction to any leaf $\mathcal{L}_{\alpha}$
has non-negative curvature. 
\end{prop}
Interestingly, in the presence of a foliation as above the closed
positive current $S:=(dd^{c}\Phi)^{n}$ on $M$ of dimension $(1,1),$
appearing in the proof of Theorem \ref{thm:main text}, can be written
as an average of the integration currents $[\mathcal{L}_{\alpha}]$
defined by the leaves of the foliation: 
\[
S=\int_{\alpha\in X}[\mathcal{L}_{\alpha}]\mu,
\]
 where $\mu:=(dd^{c}\phi_{0})^{n}$. 

Another special case of Theorem 3.4, concerning the case when the
current $S$ is assumed to be a smooth complex curve (but not necesseraily
a leaf of a foliation) and $\Phi$ is $C^{2}-$smooth has previously
appeared in connection to the problem of constructing low regularity
(i.e. not $C^{2})$ solutions to complex Monge-Amp\`ere equations (see
Lemma in \cite{b-f} and Proposition 2.2 in \cite{dar-l}).

\section{Uniqueness results}

In this section we shall show how the convexity of the K-energy implies uniqueness of metrics, up to automorphisms,  of metrics of constant scalar curvature and more generally extremal metrics. Recall that $\h(X,\omega)$ denotes the space of (smooth) potentials of K\"ahler metrics on $X$ that are cohomologous to a fixed reference metric $\omega>0$ (see the introduction). The tangent space of $\h$ is the space of smooth functions on $X$, and we can identify the space of K\"ahler metrics cohomologous to $\omega$ with $\h$ modulo constants. We will use the twisted Mabuchi functionals from section 3.1.1 and start with some preparations.

Let $\mu>0$ be a smooth volume form on $X$, that for simplicity we normalize so that
$$
\int_X d\mu =\int_X \omega^n.
$$
We have then defined the function
$$
\F_\mu(u)=\int_X ud\mu -\e(u):=I_\mu(u) -\e(u)
$$
in section 3.1.1. The basic idea is to use the twisted Mabuchi functionals
$$
\M_s:=\M+ s\F_\mu,
$$
for $0<s<<1$. 
The main difficulty in the proof is that although we know that $\M$ is convex along generalized geodesics, we don't know when it is linear along the geodesics. (Conjecturally this holds only for geodesics that come from the flow of a holomorphic vector field.)  
 Therefore we perturb $\M$ by adding $s\F_\mu$ which gives us  a strictly convex functional. In case there are no nontrivial holomorphic vector fields on $X$, one can prove by the implicit function theorem that near each critical point of $\M$ there is a critical point of $\M_s$. By strict convexity, there can be at most one critical point of $\M_s$, and it follows that there is at most one critical point of $\M$ too. In case there are holomorphic vector fields it of course no longer holds that there are critical points of $\M_s$ near each critical point of $\M$ - if it did, we would get absolute uniqueness and not just uniqueness up to automorphisms. However,  it turns out that each critical point of $\M$ can be moved by an element in Aut$_0(X)$ to a new critical point, which {\it can}  be approximated by critical points of $\M_s$, and this gives uniqueness up to automorphisms. The proof of this latter fact requires a rather sophisticated version of the implicit function theorem, so to simplify we shall instead work with 'almost critical points', which avoids the use of the implicit function theorem. .

\bigskip

With our normalization, $\F_\mu$  vanishes on constants so it decends to a functional on the space of K\"ahler forms in $[\omega]$. We have already seen that $\F_\mu$ is convex; next we shall prove that it is strictly convex in a certain sense. Since $\e$ is linear, this amounts to proving the strict convexity of $I_\mu$.
\begin{prop}
$I_{\mu}$ is strictly convex along $C^{1,1}$-subgeodesics in the sense that if $u_t$ is a $C^{1,1}$-subgeodesic and $f(t):=I_\mu(u_t)$ is affine, then $\omega_t=dd^cu_t + \omega$ is constant. More precisely, if $\omega_t=dd^cu_t+\omega\leq C \omega$  and $\mu\geq A\omega^n$, then
$$
f'(1)-f'(0)\geq \delta A/(C^{n+1})d(\omega_0,\omega_1)^2,
$$
where $\delta>0$ only depends on $\mu$, $\omega$  and $X$, and $d(\omega_0,\omega_1)$ is the Mabuchi distance. 
\end{prop}
\begin{proof} Assume first that $u_t$ is a smooth subgeodesic and $\omega_t>0$ for all $t$.  Then
$$
f''(t)=\int_X \ddot u_{tt} d\mu\geq\int_X |\dbar \dot u_t|^2_{\omega_t} d\mu,
$$
since $u_t$ is a subgeodesic. Assume $\omega_t\leq C\omega$ for all $t$. Then
$$
 |\dbar\dot u_t|^2_{\omega_t}\geq C^{-1} |\dbar\dot u_t|^2_{\omega}.
$$
Since $\omega$ and $\mu$ are fixed and $\dot u_t$ is a function we have that
$$
\int_X  |\dbar\dot u_t|^2_{\omega}d\mu\geq\delta \int_X |\dot u_t -a_t|^2 d\mu,
$$
where $a_t$ is the average of $\dot u_t$ with respect to $\mu$ and $\delta$ only depends on $\mu$, $\omega$ and $X$. Hence
$$
f''(t)\geq \delta/C\int_X |\dot u_t -a_t|^2 d\mu.
$$
Clearly it follows that $\dot u_t=a_t$  if $f$ is affine. If $u_t$ is only of class $C^{1,1}$ we can write $u_t$ as a decreasing limit of subgeodesics that converge uniformly in $C^1$ and are such that the constant $C$ can be kept fixed. It then follows that $f(t)$ also converges in $C^1$ and we get that
$$
\int_0^1 dt\int_X|\dot u_t-a_t|^2 d\mu =0,
$$
so $\dot u_t=a_t$ again. Hence $\omega_t$ is independent of $t$.

For the second statement we notice that we also have proved that
$$
f'(1)-f'(0)\geq(\delta/C)\int_0^1 dt\int_X|\dot u_t-a_t|^2 d\mu.
$$
But
$$
\int_X|\dot u_t-a_t|^2 d\mu\geq AC^{-n}\int_X|\dot u_t-a_t|^2 \omega_t^n\geq  AC^{-n}\int_X|\dot u_t-b_t|^2 \omega_t^n,
$$
where $b_t$ is the average of $\dot u_t$ with respect to $\omega_t^n$. Since
$$
\int_0^1 dt\int_X|\dot u_t-b_t|^2 \omega_t^n=d(\omega_1,\omega_0)^2,
$$
we have also proved the second statement.
\end{proof}
We will also need a lemma on how $\F_\mu$ depends on $\mu$.
\begin{lem} Let $\mu$ and $\nu$ be two smooth volume forms with total mass equal to the mass of $\omega^n$. Then 
$$
|\F_\mu(u) -\F_\nu(u)|\leq C_{\mu,\nu}
$$
for all $u$ in $\h$.
\end{lem}
\begin{proof}
By Yau's solution of the Calabi conjecture, we can write
$$
\mu=\omega_\mu^n, \quad \nu=\omega_\nu^n,
$$
with $\omega_{\mu,\nu}$ in $[\omega]$. (Of course, the proof does not really depend on the solution of the Calabi conjecture, since we could have  used only volume forms that {\it are} given as powers of K\"ahler forms in the proof.) Then $\omega_\mu-\omega_\nu=dd^c v$ for some function $v$ on $X$. Hence
$$
\F_\mu(u)-\F_\nu(u)=\int_X u(\omega_\mu^n-\omega_\nu^n)=\int_X u(dd^c v\wedge\sum \omega_\mu^{n-k-1}\wedge\omega_\nu^k).
$$
Integration  by parts gives
$$
\F_\mu(u)-\F_\nu(u)=\int_X v(dd^c u\wedge\sum \omega_\mu^{n-k-1}\wedge\omega_\nu^k)= \int_X v(\omega_u-\omega)\wedge\sum \omega_\mu^{n-k-1}\wedge\omega_\nu^k),
$$
which is clearly bounded by a constant depending only on the sup-norm of $v$ and the volume of $[\omega]$.
\end{proof}
Next we discuss briefly the Hessian of $\M$ on the space of smooth K\"ahler potentials. Denote $F=d\M$, the differential of the Mabuchi functional. It is a 1-form on $\h$ whose action on an element $v$ of the tangent space of $\h$, i e a smooth function is given by
$$
F(u).v=-\int_X v(R_{\omega_u}-\hat R_{\omega_u})\omega_u^n.
$$
The Mabuchi metric induces a connection $D$ on the tangent bundle of $\h$, which in turn induces a (dual) connection on the space of 1-forms that we also denote by $D$. If $v$ is a vector at a point $u$ we can then apply $D_v$  to the 1-form $F$ and get a new 1-form $D_vF$. By definition, if $w$ is another vector at $u$, then
$$
D_vF. w= H_\M(v,w)
$$
is the Hessian of $\M$ at $u$, which  in spite of appearances is a symmetric bilinear form (since the connection is symmetric). It is well known that this equals
$$
  H_\M(v,w)=\int_X \D_u v\, \overline{\D_u w}\,\omega_u^n,
$$
where $\D_u$ is the Lichnerowicz operator, see \cite{do00}. This is an elliptic operator of second order, $\D_u=\nabla_u\dbar$ where $\nabla$ is the Chern connection on the cotangent bundle of $X$ for the metric $\omega_u$. It is also well known that
$$
\D_u w=0
$$
if and only if the $(1,0)$ vector field $V$ (the complex gradient of $w$) on $X$ defined by
$$
V\rfloor \omega_u=i\dbar w
$$
is holomorphic. 
\begin{prop}
Let $\nu$ be a smooth volume form on $X$ that defines a 1-form $G_\nu$ at $u$ by
$$
G_\nu.w=\int_X wd\nu.
$$
Then there is a vector $v$ at $u$ such that 
$$
D_vF|_u =G_\nu
$$
if and only if $G_\nu.w=0$ for all $w$ such that the complex gradient of $w$ is holomorphic.
\end{prop}
\begin{proof}
We have
$$
D_vF|_u. w=H_\M(v,w)=\int_X  \D_u v\, \overline{\D_u w}\,\omega_u^n=\int_X \D_u^*\D_u v \,\,w \,\omega_u^n.
$$
Hence 
$$
D_vF|_u =G\nu
$$
means that
$$
\D_u^*\D_u v \, \omega_u^n =\nu.
$$
Since $\D_u^*\D_u v$ is a self adjoint  elliptic operator, this equation is solvable if and only if $\nu$ annihilates the kernel of  $\D_u^*\D_u v$, which is the same as the kernel of $\D_u$, i e the space of functions whose complex gradients are holomorphic.
\end{proof}

We are now ready for the uniqueness and we start with the case when there are no nontrivial holomorphic vector fields on $X$.
\begin{thm} Assume $\omega_{u_0}$ and $\omega_{u_1}$ are metrics of constant scalar curvature on $X$ and that $X$ has no nontrivial holomorphic vector fields. Then  $\omega_{u_0} =\omega_{u_1}$.
\end{thm}
\begin{proof} By hypthesis $u_0$ and $u_1$ are both critical points of $\M$, so
$F(u_0)=F(u_1)=0$. Let $\mu$ be a strictly positive volume form normalized as in the beginning of this section. The differential of $\F_\mu$ at $u_0$ is $G(u_0)=G_\nu$, where $\nu=\mu-\omega_{u_0}^n$. Since by our normalization this measure annihilates constants, which are now the only functions with holomorphic complex gradient,  Proposition 4.3 implies that we can solve
$$
D_{v_0}F|_{u_0}= -G(u_0).
$$
Consider the functional
$$
\M_{s}:=\M +s\F_\mu,
$$
and its differential
$$
F_s(u)=F(u) +s G(u).
$$
If $w_s$ is a smooth curve of continuous functions (that we consider as a tangent vector field along the curve $u_0 +s v_0$ in $\mathcal{H}$) we have
$$
(d/ds)|_0F_s(u_0+sv_0).w_s= D_{v_0}F|_{u_0}. w_0 +F(u_0). D_{v_0}w_s +G(u_0).w_0=0,
$$
since $F(u_0)=0$ and $D_{v_0}F|_{u_0}= -G(u_0)$. Hence $F_s(u_0+sv_0).w_s=O(s^2)$. 
More precisely, since 
$$
F_s(u_0+ sv_0).w=\int_X w dV_s
$$
for some smooth density $V_s$ depending smoothly on $s$, we get that $V_s =O(s^2)$, so
$$
|F_s(u_0+ sv_0).w|\leq C s^2\sup_X |w|.
$$
We can now do the same construction for the other critical point $u_1$ and obtain another function $v_1$ with similar properties. Then connect for $0<s<<1$ the smooth points  $u_0+sv_0$ and $u_1+sv_1$ by a $C^{1,1}$-geodesic $u_t^s$. We need to relate $F$ and $F_s$, the formal derivatives of $\M$ and $\M_s$, to the actual one sided derivatives along the geodesics at the endpoints. It is not a priori clear that they coincide since the formal derivatives are the derivatives in $\h$, the space of smooth potentials, and the weak geodesic has less regularity. However it follows from Lemma 3.5 that
$$
(d/dt)|^+_{t=0}\M(u^s_t)\geq F(u^s_0). (d/dt)|_{t=0}u^s_t.
$$
and also that we have the converse inequality at the other endpoint. Since $\F_s$ is differentiable one time on the closed geodesic, the same inequalities hold for $\M_s$ as well. 
 
Since $\M(u_t^s)$ is convex and $\e(u_t^s)$ is linear in $t$ we get that
$$
0\leq s((d/dt)|_1-(d/dt)|_0))I_{\mu}(u_t^s)\leq((d/dt)|_1-(d/dt)|_0))\M_{s}(u_t^s)\leq 
$$
$$
F(u^s_1). (d/dt)|_{t=1}u^s_t -F(u^s_0). (d/dt)|_{t=0}u^s_t\leq O(s^2).
$$
Dividing by $s$ we get that
$$
((d/dt)|_1-(d/dt)|_0))I_{\mu}(u_t^s)\leq C's,
$$
so 
by Proposition 4.1,  $d(\omega_{u_0^s},\omega_{u_1^s})^2\leq C''s$. Here we have also used the fact that the constant $C$ in Proposition 4.1 can be taken independent of $s$, i e that the $L^\infty$-bound on $\omega_{u_t^s}$ can be taken uniform in $s$, see  \cite{Berman-Demailly}.  Hence  $d(\omega_{u_0},\omega_{u_1})=0$ which implies that $\omega_{u_0}=\omega_{u_1}$ by a result of Chen, see \cite{c0}.
\end{proof}

Notice that there are two main points of the argument. Apart from the convexity along weak geodesics we also use that $\M$ is strictly convex on $\h$ (modulo constants) in the formal sense that its Hessian is strictly positive if there are no nontrivial holomorphic vecor fields. This means that the derivative of $F=d\M$ is invertible which allows us to solve $D_{v_0}F|_{u_0}= -G(u_0)$. The same principle is illustrated in the next result which concerns uniqueness of metrics of constant $\alpha$-twisted curvature (cf section 3.1.1).
\begin{thm} Let $\alpha$ be a K\"ahler form on $X$. Then there is at most one metric $\omega_0$ in a given K\"ahler class $[\omega]$ with constant $\alpha$-twisted curvature.
\end{thm}
\begin{proof} Recall that $\omega_0$ is a critical point of the twisted Mabuchi functional $\M_\alpha=\M +\e^\alpha-c(\alpha)\e$ ( cf section 3.1.1). One can check that the Hessian of $\e^\alpha$ (at a smooth point) is strictly positive, and we also have that the Hessian of $\e$ is zero since $\e$ is linear along geodesics. If we put $F_\alpha=d\M_\alpha$ it follows that we can solve $D_{v_0}F_\alpha|_{u_0}= -G(u_0)$ as in the proof of the previous theorem, and again we conclude by the convexity of $\M_\alpha$.
\end{proof}

We finally turn to the case of metrics of constant scalar curvature when there are non zero holomorphic vector fields on $X$. The argument is then essentially the same, with the additional difficulty that  we can not solve the equation
$$
D_{v_0}F|_{u_0}= -G(u_0)
$$
in general. Therefore we shall make a preliminary modification of $u_0$ by applying an automorphism in Aut$_0(X)$, so that after the modification $G(u_0)$ annihilates all functions with holomorphic complex gradient.
\begin{prop} Let $S$ be the submanifold  of $\h$ consisting of all potentials of metrics $g^*\omega_{u_0}$, where $g$ ranges over Aut$_0(X)$. Then $\F_\mu$ has a minimum and hence a critical point on $S$. This implies that $G=d\mathcal{F}_\mu$ annihilates all real functions whose complex gradients are holomorphic.
\end{prop}
\begin{proof} Any holomorphic vector field $V$ determines a geodesic ray starting at $u_0$ obtained by following the flow of $V$, and $S$ is the union of all such rays. If $\mu=\omega_{u_0}^n$, then $u_0$ is a critical point of $\F_\mu$. Since $\F_\mu$ is strictly convex along each ray it follows that $\F_\mu$ is proper on each ray if $\mu=\omega_{u_0}^n$. Since $S$ is of finite dimension, it follows that $\F_\mu$ is proper on $S$ in this case. By lemma 4.2 this implies that $\F_\mu$ is proper on $S$ for any choice of $\mu$, and so must have a minimum. For the last claim we let $g_t=\exp(tV)$ be the ray determined by a holomorphic field $V$. Then the Lie derivative of $\omega$ with respect to $V$ equals
$$
dV\rfloor \omega=i\ddbar h^V_\omega
$$
and also
$$
(d/dt)g_t^*\omega=(d/dt)(\omega +i\ddbar u_t) = i\ddbar\dot u_t.
$$
Hence $h^V_\omega=\dot u_t|_{t=0}$ so $d\mathcal{F}_\mu.h^V_\omega=d\mathcal{F}_\mu.\dot u_t=0$ if $\omega$ is a critical point of $\mathcal{F}_\mu$.
\end{proof}

By Proposition 4.3 this implies that we can find a $v_0$ that solves 
$$
D_{v_0}F|_{u_0}= -G(u_0).
$$
We now apply this when $u_0$ is a critical point of $\M$. Notice that $\M$ is  invariant under the action of Aut$_0(X)$, since it is linear along the flow of holomorphic vector fields and is bounded from below, cf Corollary 1.2.  Hence the point we get after applying the automorphism is still a critical point of $\M$. If $u_1$ is another critical point  we can apply the same argument to $u_1$. The proof of Theorem 4.4 then applies without change and we see that after applying these automorphisms  $\omega_{u_0}=\omega_{u_1}$. Therefore we have proved
\begin{thm} Assume that $\omega_{u_0}$ and $\omega_{u_1}$ are metrics of constant scalar curvature. Then there is an automorphism $g$ in Aut$_0(X)$  such that
$$
g^*(\omega_{u_1})=\omega_{u_0}.
$$
\end{thm}
{\it Remark.} This argument follows an idea by Bando and Mabuchi in \cite{b-m} to prove  uniqueness of K\"ahler-Einstein metrics and it may be illuminating to compare the arguments.  Bando and Mabuchi consider another perturbation defined by Aubin's continuity method which applies when $[\omega]= c_1(X)$ and can be written
$$
\mbox{Ric}\omega_s =(1-s)\omega_s +s\alpha
$$
where $\alpha$ is a fixed K\"ahler form. Using a bifurcation technique that plays the role of our Proposition 4.6, they show that a particular choice $\omega_0$
in  the Aut$_0$-orbit of a K\"ahler-Einstein metric extends to a smooth curve of solutions to the perturbed equations above. Using a priori estimates, they show that this curve extends to a smooth curve for $s$ in $[0,1]$. For $s=1$ it is easy to see that the perturbed equation has a unique solution, and it then follows from the invertibility of the linearized equations that we have uniqueness for $s=0$ as well. One simplifying feature of our argument, which is based on convexity, is that it is enough to consider small, in fact even infinitesimal, perturbations of the original problem.\qed
\subsection{\label{sub:Calabi's-extremal-metrics}Calabi's extremal metrics}

The\emph{ extremal} K\"ahler metrics (in a given K\"ahler class) introduced
by Calabi \cite{ca-2}, generalizing constant scalar curvature metrics,
are defined as the critical points of the $L^{2}-$norm of the scalar
curvature, i.e. the functional $\omega\mapsto\int_{X}R_{\omega}^{2}\omega^{n}$
on the space of K\"ahler metrics in a the fixed K\"ahler class. As shown
by Calabi this equivalently means that the gradient  of $R_{\omega}$
is a holomorphic vector field, or more precisely that the $(1,0)$ field $V$ with real part equal to the gradient of $R_\omega$ is holomorphic. We shall now generalize Theorem 4.7 to extremal K\"ahler metrics. This builds on the fundamental work in \cite{fu-ma}, \cite{gu} and \cite{si}, which for completeness we develop from scratch in a form suitable in this context.

The holomorphic vector field $V$ will in general depend on the extremal metric. The first step in the proof, following \cite{fu-ma}, is to prove that one can obtain a unique 'extremal vector field' by fixing a compact subgroup $K$ of Aut$_0(X)$ and requiring that the flow of $\Im V$ lie in $K$. Once this is done we, following \cite{gu} and  \cite{si},  modify the Mabuchi functional to obtain another functional $\M_V$, defined on $\h_V$, the space of K\"ahler metrics invariant under $\Im V$, by adding a term $\e_V$, depending on $V$. The  extremal metrics corresponding to the now fixed field $V$ are now  critical points of $\M_V$ on $\h_V$. The energy  functional $\e_V$ is linear, so $\M_V$ is also convex along geodesics in $\h_V$. Given all this, the proof of the uniqueness of extremal metrics follows the same lines as before.

We start with a few preparations. 
Let $\omega$ be any K\"ahler form on $X$. Recall that if $h$ is a complex valued function on $X$ we define a vector field of type $(1,0)$ by 
$$
V\rfloor \omega=i\dbar h.
$$ 
$V=:\nabla_\omega h$ is called the complex gradient of $h$ and we have that
$$
2\Im V\rfloor\omega= d^c\Im h +d\Re h.
$$
(Contrary to our earlier conventions we here write $d^c$ for $i(\dbar-\partial)$.)
Therefore we see that the Lie derivative of $\omega $ along $\Im V$, $L_V\omega=d \Im V\rfloor\omega= (1/2)dd^c\Im h$ vanishes if and only if $\Im h$ is a constant and in that case $\Re h/2$ is a  Hamiltonian of $\Im V$. Then  the real part of $V$ is  the real gradient of $h/2$, so with $h=R_\omega$ we see that $\omega$ is extremal if and only if the complex gradient of $R_\omega$ is holomorphic. We  normalize by choosing $h$ so that
$$
\int_X h\omega^n=0.
$$
Then $h$ is uniquely determined by $V$ and $\omega$, and we write $h=h^V_\omega$. Note that with this normalization, $\omega$ is invariant under the flow of $\Im V$ if and only if $h^V_\omega$ is real valued. Denote by $H(X)$ the space of holomorphic vector fields that arise as complex gradients. Note that $h^{V+W}_\omega=h^V_\omega +h^W_\omega$ and we also have
\begin{lem} If $\omega_u= \omega_0+i\ddbar u$ and $V\in H(X)$, then
$$
h^V_{\omega_u}=h^V_{\omega_0}+ V(u).
$$
\end{lem}
In the proof of this we shall use  a technical lemma that will reappear later several times.
\begin{lem} If $u$ and $v$ are functions on $X$
$$
n\int_X v\,i\ddbar u\wedge\omega^{n-1}=-\int_X \nabla_\omega v(u) \,\omega^n.
$$
\end{lem}
\begin{proof} This follows from integration by parts and noting that $\nabla_\omega v(u)=\langle \partial u,\partial \bar v\rangle_\omega$.
\end{proof}
To prove Lemma 4.8 we first note that $i\dbar(h^V_{\omega_0}+ V(u))=V\rfloor\omega_u$, so $h^V_{\omega_u}=h^V_{\omega_0}+ V(u)+c(u)$, where $c(u)$ is constant on $X$. Moreover
$$
0=(d/dt)\int_X h^V_{\omega_{tu}}\omega_{tu}^n=\int_X (V(u)+\dot c(tu))\omega_{tu}^n+n\int_X h^V_{\omega_{tu}}i\ddbar u\wedge\omega_{tu}^{n-1}.
$$
By the technical lemma 4.8, $\dot c(tu)=0$, so $c(u)=0$ since $c(0)=0$.

We next, following \cite{fu-ma},  define a bilinear form on $H(X)$ by
$$
\langle V,W\rangle_\omega =\int_X h^V_\omega \, h^W_\omega \omega^n.
$$
\begin{prop} $\langle \, ,\, \rangle_\omega$ only  depends on the cohomology class $[\omega]$.
\end{prop}
\begin{proof} We take a curve of metrics $\omega_t=\omega+i\ddbar u_t$ in $[\omega]$ and differentiate:
$$
(d/dt)\int_X h^V_{\omega_t} \, h^W_{\omega_t} \omega_t^n=
\int_X \left (V(\dot u)h^W_{\omega_t}+W(\dot u)h^V_{\omega_t}\right)\omega_t^n+
n\int_X h^V_{\omega_t} \, h^W_{\omega_t} i\ddbar\dot u\wedge\omega_t^{n-1}.
$$
By the technical lemma, this expression vanishes which proves the proposition.
\end{proof}
Since the cohomology class is fixed in our discussion we can thus consider the form as fixed and write $\langle\,,\,\rangle_\omega=\langle\,,\,\rangle$. 

Let now $K$ be a compact subgroup of Aut$_0(X)$ and denote by $H_K(X)$ the subspace of $H(X)$ consisting of holomorphic vector fields $V$ such that the flow of $\Im V$ lies in $K$. 
\begin{prop} For any compact subgroup $K$ of Aut$_0(X)$ the restriction of $\langle\,,\,\rangle$ to $H_K(X)$ is real valued and positive definite; in particular non degenerate.
\end{prop}
\begin{proof} Taking averages of an arbitrary K\"ahler form in our class, we can represent our form by a $K$-invariant K\"ahler form $\omega$. Then $h^V_\omega$ is real valued if $V$ lies in $H_K$. Both claims of the proposition follow directly from this.
\end{proof}
For a holomorphic vector field $V$ in $H(X)$ we  put
$$
\mathcal{C}^\infty_V=\{u\in \mathcal{C}^\infty(X;\R); \Im(V)u=0\}
$$
and denote by Aut$_0(X,V)$  the  subgroup  of Aut$_0(X)$ of automorphisms commuting with the flow of $V$.

\begin{prop} Let $V$ be a vector field in $H(X)$ and $\omega_0$ a K\"ahler form invariant under the flow of $\Im V$. Then a real valued function $u$ lies in $\mathcal{C}^\infty_V$ if and only if the vector fields $\Im \nabla_{\omega_0}u$ and $\Im V$ commute. If moreover $\nabla_{\omega_0}u$ is holomorphic,  then $W$ lies in the Lie algebra of Aut$_0(X,V)$ .
\end{prop}
\begin{proof}
For $v$ real valued, let $W_v=2\Im\nabla_{\omega_0}v$ be the vector field determined by the Hamiltonian $v$ and the symplectic form $\omega_0$, $W_v\rfloor \omega_0=dv$. Then
$$
[W_u, \Im V]= W_{\{u,h^V_{\omega_0}\}}
$$
(where $\{\,,\,\}$ is the Poisson bracket).
Since for any $u$ and $v$, $\{u,v\}=W_vu$ we see that $\{u,h^V_{\omega_0}\}=2\Im V u$, so $\Im V u=0$ if and only if $W_u$ and $\Im V$ commute. If we also assume that $\nabla_{\omega_0}u$ is holomorphic, then $\Im \nabla_{\omega_0}u$ and $\Im V$ commute if and only if  $\nabla_{\omega_0}u$ and $V$ commute, which means that  $\nabla_{\omega_0}u$ lies in the Lie algebra of Aut$_0(X,V)$ .
\end{proof}

Finally, given a field $V$ in $H(X)$ we define an associated energy functional $\e_V$ by letting
$$
d\e_V|_{\omega}.\dot u:= \int_X \dot u \, h^V_\omega\, \omega^n.
$$
The next proposition shows that this indeed defines a function on the  subspace $\h_V$ of $\h$ consisting of K\"ahler metrics invariant under $\Im V$ , and also computes its second derivative along a curve.
\begin{prop} Let $\omega_u=\omega_0 +i\ddbar u$ where $u\in \mathcal{C}^\infty_V$ depends smoothly on two real parameters $s$ and $t$. Assume that $\omega_0$ is invariant under $\Im V$.Then
$$
(d/ds)\int_X \dot u_t h^V_{\omega_u}\omega_u^n=
\int_X (\ddot u_{s t}-( \partial \dot u_t,\partial\dot u_s)_{\omega_u} ) h^V_{\omega_u} \omega_u^n,
$$
where $ ( \,,\,)_{\omega_u}=\Re \langle \,,\,\rangle_{\omega_u}$ is the real scalar product defined by $\omega_u$.
\end{prop}
\begin{proof} A direct computation using the technical lemma shows that
$$
(d/ds)\int_X \dot u_t h^V_{\omega_u}\omega_u^n=
\int_X (\ddot u_{s t}-\langle \partial \dot u_t,\partial\dot u_s\rangle_{\omega_u} ) h^V_{\omega_u} \omega_u^n.
$$
If $u$ lies in $\mathcal{C}^\infty_V$ then $h^V_{\omega_u}=h^V_{\omega_u} +V(u)$ is real valued for all $s$ and $t$. Hence the proposition follows by taking real parts.
\end{proof}
Since this expression is symmetric in $s$ and $t$ it follows that 
$$
\int_0^1 dt\int_X\dot u_t h^V_{\omega_u}\omega_u^n
$$
is independent of the choice of path between $u_0=0$ and $u_1$, so $\e_V$ is a well defined function.  Note  also that  $d\e_V.\dot u_t$ vanishes if $\dot u_t$ is a constant, so $\e_V(u)$ decends to a function on the space of K\"ahler forms in $[\omega]$.
In addition, we see from Proposition 4.13 that
$$
(d/dt)^2\e_V(u)=\int_X (\ddot u_{tt}-|\dbar\dot u_t|^2_{\omega_u})h^V_{\omega_u} \omega_u^n.
$$
This formula extends to curves  in $\mathcal{C}^{1,1}_\C$, if we {\it define} $h^V_{\omega_u}=h^V_{\omega_0}+V(u)$, simply by approximation. Thus we see that $\e_V$ is linear along a $\mathcal{C}^{1,1}_\C$-geodesic in $\h_V$. For any pair of metrics in $\h_V$ the weak geodesic between them will remain in $\h_V$, so $\e_V$ is linear along the connecting geodesic. From this we also conclude that  the Hessian of $\e_V$ at $\omega_0$ restricted to 
 $\mathcal{C}^\infty_V$ vanishes.

We are now ready for the proof of the uniqueness of extremal metrics.
Following  \cite{fu-ma} we shall first see that the holomorphic
vector field that arises as the complex gradient of the scalar curvature $R_\omega$ of an extremal metric  is uniquely determined by the given K\"ahler
class, modulo $\mbox{Aut}\ensuremath{_{0}(X)}:$
\begin{prop} Let $K$ be a a maximal compact subgroup of Aut$_0(X)$  and let $\omega_0$ be an extremal metric in $[\omega]$. Let $V_0$ be the associated vector field $V_0=\nabla_{\omega_0}R_{\omega_0}$. Then

\medskip

\noindent 1. There is an element $g$ in Aut$_0(X)$ such that after replacing $\omega_0$ by $g^*\omega_0$ the flow of $\Im V_0$ lies in $K$,

and

\noindent 2. If $\omega_1$ is another extremal metric in the same cohomology class, with associated vector field $V_1$ such that the flow of $\Im V_1$ also lies in $K$, then $V_0=V_1$.
\end{prop}
Hence, given $K$, we may speak of 'the' extremal vector field.
\begin{proof}
Since $V_0$ is the complex gradient of a real valued function $R_{\omega_0}$, the flow of $\Im V_0$ is an isometry as we have seen. Hence the flow of $\Im V_0$ lies in some maximal compact group $K_0$. By a fundamental theorem of Iwasawa, \cite{Iwasawa}, the two groups $K$ and $K_0$ are conjugate under some automorphism $g$. This proves 1. 

 Let now $V_0$ be the holomorphic vector field associated to $\omega_0$. Then $V_0$ lies in $H_K(X)$.  Let $W$ be an arbitrary field in $H_K$. Then 
$$
-\langle V_0, W\rangle=\int_X (R_{\omega_0}-\hat R_{\omega_0})h^W_{\omega_0}\omega_0^n.
$$ 
By definition, this is nothing but the negative of the Futaki invariant of $W$, \cite{fu}, which is well known  not to depend on the choice of K\"ahler metric. In particular it also equals $ -\langle V_1, W\rangle$, so since the bilinear form is non degenerate on $H_K$, it follows that $V_0=V_1$.
\end{proof}
The main theorem of this section generalizes Theorem 4.7.
\begin{thm} 
Given any two extremal K\"ahler metrics $\omega_{0}$
and $\omega_{1}$ in a given cohomology class there exists an element
$g\in\mbox{Aut}_{0}(X)$ such that $\omega_{0}=g^{*}\omega_{1}.$ \end{thm}
Following \cite{gu} and \cite{si} we   modify the Mabuchi funcional to obtain another functional which has our extremal metrics as critical points. By  Proposition 4.13  we may assume  that the vector fields associated to $\omega_0$ and $\omega_1$ are the same field $V$. Then both $\omega_0$ and $\omega_1$ are invariant under $\Im V$ and  hence invariant under  the closure of the one parameter subgroup of Aut$_0(X)$ generated by $\Im V$, which we call $T$.  Let $\M_V:=\M+\e_V$, where $\e_V$ is the previously introduced energy functional associated to the extremal field $V$. $\M_V$ is defined on the subspace $\h_V$ of $\h$ of K\"ahler potentials invariant under $\Im V$. Then booth $\omega_0$ and $\omega_1$ are critical points of $\M_V$ on $\h_V$. We now let $\mu$ be a smooth $T$-invariant volume form, normalized as before and consider, following the proof of Theorem 4.7,  the functionals
$$
\M_V +s\F_\mu,
$$
where $s$ is a small positive number and let $F_V(u,s):= d(\M_V +s\F_\mu)|_u$. We shall prove that if $\omega_0=\omega+i\ddbar u_0$ then there exists a smooth function $v_0$ such that $F_V(u_0+sv_0,s)= O(s^2)$ and as before this amounts to solving the equation
$$
D_{v_0}d\M_V|{u_0}=-d\F_\mu|_{\omega_0}=-(\mu-\omega_0^n).
$$ 
Moreover, we look for $v_0$ such that $\Im V(v_0)=0$. We proceed as in the proof of Theorem 4.7, but this time we first replace $\omega_0$ by $g^*\omega_0$ where $g\in$ Aut$_0(X,V)$     is chosen to give the minimum of $\F_\mu$ on the orbit $\mbox{Aut}_0(X,V) \omega_0$, i e we use the subgroup Aut$_0(X,V)$  instead of the full group Aut$_0(X)$. Notice that $\M_V$  is invariant under the action of Aut$_0(X,V)$  by the same reason as before: It is linear along the flow of vector fields that commute with $V$ and is bounded from below on $\h_V$ (this can be proved in the same way that we proved Corollary 1.2). Therefore $g^*\omega_0=:\omega_0'$ is still critical for $\M_V$.

Then $d\F_\mu|_{\omega_0'}$ annihilates all real valued  functions whose complex gradients lie in the Lie algebra of Aut$_0(X,V)$ , cf the proof of Proposition 4.6. By Proposition 4.11 it follows that $d\F_\mu$ annihilates all functions in $\mathcal{C}^\infty_V$ with holomorphic complex gradients. But, if $h$ is a general real valued function with holomorphic complex gradient, and Av$_{T}(h)$ denotes the average of $h$ over $T$, then (since $\mu -\omega_0^n$ is $T$-invariant)
$$
\int_X h(\mu-\omega_0^n)=\int_X \mbox{Av}_{T} (h)(\mu-\omega_0^n)=0,
$$
since Av$_{T}(h)$ is annihilated by $\Im V$. Hence $\mu -\omega_0^n$ annihilates all real functions with holomorphic complex gradient, which by Proposition 4.3 means that we can solve 
$$
-D_{v_0}d\M|_{u_0}=\mu-\omega_0^n.
$$
Replacing $v_0$ by its average over $T$ we can also find a solution that is $T$-invariant, i e annihilated by $\Im V$. Finally, we recall that by our formula for the second derivative of $\e_V$, the Hessian of $\e_V$ restricted to $\mathcal{C}^\infty_V$ vanishes, so we have also solved
$$
-D_{v_0}d\M_V|_{u_0}=\mu-\omega_0^n.
$$ 
The proof is then completed in the same way as before: After applying an element of Aut$_0(X,V)$  to $\omega_1$ we may  solve in the same way
$$
-D_{v_1}d\M_V|_{u_1}=\mu-\omega_1^n.
$$ 
We then let $u_0^s=u_0+sv_0$ and $u_1^s=u_1+sv_1$ and connect with a geodesic $u^s_t$. By uniqueness the geodesics lie in $\h_V$ so $\e_V(u^s_t)$ is linear in $t$. It then follows again from Proposition 4.1 that the square of the distance between $\omega_{u_0^s}$ and $\omega_{u_1^s}$ is of order $s$, and hence
$\omega_0=\omega_1$.

\end{document}